\documentclass[12pt]{amsart}
\usepackage[colorlinks=true,citecolor=blue,linkcolor=blue]{hyperref}
\usepackage{latexsym,amsmath,amssymb}
\usepackage{accents}
\usepackage{a4wide}
\usepackage{color}
\usepackage[colorinlistoftodos,prependcaption,textsize=tiny]{todonotes}
  
\title[fractional div-curl quantities]{\protect{Fractional div-curl quantities and applications to nonlocal geometric equations}}

\author{Katarzyna Mazowiecka}
\author{Armin Schikorra}

\address[Katarzyna Mazowiecka]{Mathematisches Institut, 
Abt. f\"ur Reine Mathematik, 
Albert-Ludwigs-Universit\"at,
Eckerstra\ss{}e 1,
79104 Freiburg im Breisgau,
Germany
\newline
\& Institute of Mathematics,
University of Warsaw,
Banacha 2,
02-097 Warszawa, Poland}
\email{katarzyna.mazowiecka@math.uni-freiburg.de}

\address[Armin Schikorra]{Mathematisches Institut, 
Abt. f\"ur Reine Mathematik, 
Albert-Ludwigs-Universit\"at,
Eckerstra\ss{}e 1,
79104 Freiburg im Breisgau,
Germany}
\email{armin.schikorra@math.uni-freiburg.de}

plus 6pt minus 12pt
\def\eps{\varepsilon}

\def\vp{\varphi}

\def\N{{\mathbb N}}
\def\n{{\mathcal N}}

\def\S{{\mathbb S}}

\newtheorem{theorem}{Theorem}

\newtheorem{lemma}[theorem]{Lemma}
\newtheorem{corollary}[theorem]{Corollary}
\newtheorem{proposition}[theorem]{Proposition}

\theoremstyle{definition}
\newtheorem{remark}[theorem]{Remark}

\def\curl{{\rm curl\,}}

\newcommand{\R}{\mathbb{R}}

\newcommand{\Z}{\mathbb{Z}}

\newcommand{\brac}[1]{\left (#1 \right )}

\newcommand{\Ep}{\bigwedge\nolimits}
\newcommand{\norm}[1]{\left\|{#1}\right\|}

\newcommand{\barint}{
\rule[.036in]{.12in}{.009in}\kern-.16in \displaystyle\int }

\newcommand{\barcal}{\mbox{$ \rule[.036in]{.11in}{.007in}\kern-.128in\int $}}

\def\mvint_#1{\mathchoice
          {\mathop{\vrule width 6pt height 3 pt depth -2.5pt
                  \kern -8pt \intop}\nolimits_{\kern -3pt #1}}%
          {\mathop{\vrule width 5pt height 3 pt depth -2.6pt
                  \kern -6pt \intop}\nolimits_{#1}}%
          {\mathop{\vrule width 5pt height 3 pt depth -2.6pt
                  \kern -6pt \intop}\nolimits_{#1}}%
          {\mathop{\vrule width 5pt height 3 pt depth -2.6pt
                  \kern -6pt \intop}\nolimits_{#1}}}

\numberwithin{theorem}{section} \numberwithin{equation}{section}

\renewcommand{\div}{\operatorname{div}}

\newcommand{\lap}{\Delta }
\newcommand{\aleq}{\precsim}

\newcommand{\aeq}{\approx}

\newcommand{\laps}[1]{(-\lap)^{\frac{#1}{2}}}
\newcommand{\lapv}{(-\lap)^{\frac{1}{4}}}
\newcommand{\laph}{(-\lap)^{\frac{1}{2}}}

\begin{document}

\sloppy
\sloppy


\begin{abstract}
We investigate a fractional notion of gradient and divergence operator. We generalize the div-curl estimate by Coifman--Lions--Meyer--Semmes to fractional div-curl quantities, obtaining, in particular, a nonlocal version of Wente's lemma.

We demonstrate how these quantities appear naturally in nonlocal geometric equations, which can be used to obtain a theory for fractional harmonic maps analogous to the local theory.
Firstly, regarding fractional harmonic maps into spheres, we obtain a conservation law analogous to Shatah's conservation law and 
give a new regularity proof analogous to H\'elein's for harmonic maps into spheres. 

Secondly, we prove regularity for solutions to critical systems with nonlocal antisymmetric potentials on the right-hand side. Since the half-harmonic map equation into general target manifolds has this form, as a corollary, we obtain a new proof of the regularity of half-harmonic maps into general target manifolds following closely Rivi\`{e}re's celebrated argument in the local case.

Lastly, the fractional div-curl quantities provide also a new, simpler, proof for H\"older continuity of $W^{s,n/s}$-harmonic maps into spheres and we extend this to an argument for $W^{s,n/s}$-harmonic maps into homogeneous targets. This is an analogue of Strze\-lecki's and Toro--Wang's proof for $n$-harmonic maps into spheres and homogeneous target manifolds, respectively.
\end{abstract}

\maketitle
\newpage
\tableofcontents

\section{Introduction}
Products of divergence-free and curl-free vector fields, the so-called div-curl-quantities, play a fundamental role in Geometric Analysis. They appear, for example, in the theory of compensated compactness in the form of the div-curl Lemma: 
let $L^2(\Ep^1 \R^n)$ be the $L^2$-space of $1$-forms on $\R^n$, or equivalently the space of vector fields $L^2(\R^n,\R^n)$. Given two sequences $\{F_k\}_{k \in \N}, \{G_k\}_{k \in \N}$ in $L^2(\Ep^1 \R^n)$ which weakly converge in $L^2(\Ep^1 \R^n)$ to $F$ and $G$, respectively. In general, there is no reason that the product converges
\begin{equation}\label{eq:ccconvergence} F_k \cdot G_k \xrightarrow{k \to \infty} F\cdot G \quad \mbox{in }\mathcal{D}'(\R^n).\end{equation}
If we know, however, that (in distributional sense) $\div(F_k) = 0$ and $\curl(G_k) = 0$, or more generally assuming compactness of $\div(F_k)$ and $\curl(G_k)$ in $H^{-1}$, then \eqref{eq:ccconvergence} indeed holds true.
This phenomenon is known as compensated compactness and its theory was developed by Murat and Tartar in the late seventies \cite{Murat-1978,Murat-1981,Tartar-1978,Tartar-1979,Tartar-1982}, see also the more recent \cite{Briane-Casado-Diaz-Murat-2009, Conti-Dolzmann-Mueller}. 

In \cite{CLMS-1993} Coifman, Lions, Meyer, and Semmes found a relation between div-curl quantities and the Hardy space $\mathcal{H}^1(\R^n)$ (for a definition see Section~\ref{s:hardyspaceest}). 

\begin{theorem}[Coifman--Lions--Meyer--Semmes]\label{th:clms}
Let $F \in L^p(\Ep^1 \R^n)$ and $g \in \dot{W}^{1,p'}(\R^n)$ where $p \in (1,\infty)$ and $p' = \frac{p}{p-1}$. Then, if 
\[
 \div F = 0,
\]
we have
\[
 F \cdot dg \in \mathcal{H}^1(\R^n)
\]
with the estimate
\[
 \|F \cdot d g\|_{\mathcal{H}^1(\R^n)} \aleq \|F\|_{L^p(\Ep^1 \R^n)}\ \|d g\|_{L^{p'}(\Ep^1 \R^n)}.
\]
\end{theorem}
This theorem can be applied to all div-curl quantities, since a curl-free $G \in L^{p'}(\Ep^1 \R^n)$ can be written as $G = d g$. Theorem~\ref{th:clms} had a fundamental impact, in particular, on the regularity theory for such objects as surfaces of prescribed mean curvature and harmonic maps into manifolds \cite{Helein90,Helein91-sym,Helein91,Bethuel-1992,Riviere-2007}. We will detail a few of those results below.
\subsection*{Harmonic maps into spheres}
As a first example, consider H\'{e}lein's proof \cite{Helein90} for continuity of harmonic maps from $\R^2$ into a round sphere $\S^{N-1} \subset \R^N$. These are solutions $u \in \dot{W}^{1,2}(\R^2,\R^{N})$ that pointwise a.e. satisfy $|u| \equiv 1$ and
\begin{equation}\label{eq:harmmapeq}
 -\lap u^i = u^i\, |du|^2 \quad \mbox{$i = 1,\ldots,N$}.
\end{equation}
Henceforth, we shall use Einstein's summation convention. Since  $|u| \equiv 1$ we find \[u^k du^k = \frac{1}{2} d|u|^2 = 0.\] Thus, one can rewrite \eqref{eq:harmmapeq} as
\[
 -\lap u^i = \Omega_{ik}\cdot du^k,
\]
where
\begin{equation}\label{eq:omegaharmsphere}
 \Omega_{ik} := u^i du^k-u^k du^i.
\end{equation}
Shatah discovered in \cite{Shatah-1988}, that \eqref{eq:harmmapeq} is equivalent to the conservation law
\begin{equation}\label{eq:heldiv0}
 \div \Omega_{ik}  = 0 \quad \mbox{for }i,k \in \{1,\ldots,N\}.
\end{equation}
That is, in view of Theorem~\ref{th:clms}, \eqref{eq:harmmapeq} actually implies
\[
 \lap u^i \in \mathcal{H}^1(\R^2) \quad \mbox{for any $i = 1,\ldots,N$}.
\]
Then Calderon--Zygmund theory implies that $u$ is continuous, see \cite{Semmes-1994}. 

\subsection*{Harmonic maps into general target manifolds}
Let $\mathcal{N} \subset \R^N$ be a smooth, compact manifold without boundary. Harmonic maps into $\mathcal{N}$ are solutions $u \in W^{1,2}(\R^2,\mathcal{N})$ to
\begin{equation}\label{eq:harmmapeqperp}
 -\lap u \perp T_u \mathcal{N}.
\end{equation}
Regularity for harmonic maps from $2$-dimensional domains into general manifolds $\mathcal{N}$ was proven by H\'elein in \cite{Helein91}. Rivi\`{e}re observed in the seminal work \cite{Riviere-2007} that this equation, and in fact the Euler--Lagrange equations of a huge class of conformally invariant variational functionals, have the form
\begin{equation}\label{eq:confoeq}
 -\lap u^i = \Omega_{ij} \cdot \nabla u^j, \quad \mbox{for $i =1,\ldots,N$},
\end{equation}
where $\Omega$ is an antisymmetric $L^2$-vector field, $\Omega_{ij} = - \Omega_{ji} \in L^2(\Ep^1 \R^2)$. By an adaptation of techniques due to Uhlenbeck \cite{Uhlenbeck-1982}, see also \cite{Schikorra-2010}, he then constructed a gauge $P \in \dot{W}^{1,2}(\R^2,SO(N))$, such that
\begin{equation}\label{eq:gaugelocal}
 \div ((\nabla P) P^T - P \Omega P^T) = 0.
\end{equation}
Then for $\Omega_P := (\nabla P) P^T - P \Omega P^T$,
\[
 \div(P_{ij} \nabla u^j) = P_{jk}\ (\Omega_P)_{ij} \cdot \nabla u^k.
\]
Thus, up to a multiplicative $P_{jk}$, the right-hand side becomes a div-curl quantity. Continuity of $u$ then follows essentially from Theorem~\ref{th:clms}, cf. \cite{Riviere-Struwe-2008}.
\subsection*{$n$-harmonic maps into homogeneous target manifolds}
H\'elein's regularity argument for harmonic maps into spheres was extended to $n$-harmonic maps from $\R^n$ into spheres $\S^{N-1}$, see \cite{Fuchs-1993,Takeuchi-1994,Strzelecki-1994}. We follow Strzelecki's work \cite{Strzelecki-1994}. Take an $n$-harmonic map into a sphere $\S^{N-1}$, i.e., a solution $u \in \dot{W}^{1,n}(\R^n,\S^{N-1})$ to
\begin{equation}\label{eq:nharmmapeq}
 -\div(|du|^{n-2} du^i) = u^i\, |du|^n,  \quad \mbox{for $i = 1,\ldots,N$}.
\end{equation}
This can be rewritten, for $\Omega_{ij}$ as in \eqref{eq:omegaharmsphere},
\begin{equation}\label{eq:nharmmapeqrewritten}
 -\div(|du|^{n-2} du^i) =  |du|^{n-2} \Omega_{ij}\ du^j, \quad \mbox{for $i = 1,\ldots,N$}.
\end{equation}
Again, one can observe that \eqref{eq:nharmmapeq} is equivalent to
\[
 \div(|du|^{n-2} \Omega_{ij}) = 0.
\]
Thus the right-hand side of \eqref{eq:nharmmapeqrewritten} is again a div-curl quantity and regularity follows again essentially by Theorem~\ref{th:clms}. 

Strzelecki's argument, in turn, was generalized to homogeneous spaces by Toro and Wang~\cite{Toro-Wang}.
Let us remark here, that the regularity of $n$-harmonic maps into a general target manifold $\mathcal{N}$ is still open, cf.~\cite{Schikorra-Strzelecki-2016}.
\subsection*{Outline of the article}
In Section~\ref{s:fracclms} we will introduce a fractional version of divergence and obtain a fractional analogue of Theorem~\ref{th:clms}, see Theorem~\ref{th:clmstypestimate} below.

As we shall see, this fractional divergence and, in particular, fractional div-curl quantities, appear naturally in the theory of fractional harmonic maps and critical systems with nonlocal antisymmetric potential. We give several examples of consequences of Theorem~\ref{th:clmstypestimate} in regularity theory:
in Section~\ref{s:halfharmsphere} we consider half-harmonic maps into the sphere. Applications to critical systems with nonlocal antisymmetric potential on the right-hand side will be treated in Section~\ref{s:antisym}. The case of $W^{s,p}$-harmonic maps into round target manifolds is the subject of Section~\ref{s:homo}.

Finally, in Section~\ref{s:hardyspaceest} we give the proof of the fractional $div$-$curl$ theorem, Theorem~\ref{th:clmstypestimate}.

\section{Fractional divergence and div-curl lemmas}\label{s:fracclms}
Let us remark, without going into details, that the notion ``$s$-gradient'' $d_s$ and ``$s$-divergence'' $\div_s$ defined below can be justified by an abstract theory on Dirichlet forms $\mathcal{E}_{s,2}(f) := [f]_{W^{s,2}(\R^n)}^2$ acting on the Sobolev space $W^{s,2}(\R^n)$, see \cite[Examples 4.1]{Hinz-2015}. Here for $s \in (0,1)$, the Gagliardo seminorm $[f]_{W^{s,p}(\R^n)}$ is given by
\[
 [f]_{W^{s,p}(\R^n)} := \brac{ \int_{\R^n}\int_{\R^n}\brac{\frac{|f(x)-f(y)|}{|x-y|^{s}}}^p\ \frac{dx\ dy}{|x-y|^n}}^{\frac{1}{p}}.
\]
We will denote by $\mathcal{M}(\R^n)$ the set of all functions measurable with respect to the Lebesgue measure $dx$. Furthermore, the space of measurable (off diagonal) vector fields $\mathcal{M}(\Ep^1_{od} \R^n)$ is the space of functions $F: \R^n \times \R^n \to \R$ measurable with respect to the measure $\frac{dx\ dy}{|x-y|^{n}}$. Here ``od'' stands for ``off diagonal''. Observe that we do not require antisymmetry $F(x,y) = -F(y,x)$. On the space of measurable vector fields we define the scalar product. For two vector fields $F,G \in \mathcal{M}(\Ep^1_{od} \R^n )$  set
\[
 \langle F,G \rangle(x) \equiv F \cdot G(x) = \int_{\R^n} F(x,y)\ G(x,y)\ \frac{dy}{|x-y|^{n}}.
\]
In particular we denote
\[
 \|F\|_{2}(x) := \sqrt{\langle F,F \rangle(x)},
\]
and more generally,  
\[
 \|F\|_{p}(x) :=\brac{\int_{\R^n} |F(x,y)|^p\ \frac{dy}{|x-y|^{n}}}^{\frac{1}{p}}.
\]
The \emph{$s$-gradient} $d_s: \mathcal{M}(\R^n) \to \mathcal{M}(\Ep^1_{od} \R^n)$ acting on functions $g: \R^n \to \R$ takes the form
\[
 d_s g(x,y) := \frac{g(x)-g(y)}{|x-y|^s} \in \mathcal{M}\brac{\Ep^1_{od}\R^n}.
\]
In particular,
\[
 F \cdot d_s g(x) = \int_{\R^n} F(x,y)\ \frac{g(x)-g(y)}{|x-y|^s}\ \frac{dy}{|x-y|^n}.
\]
We will call the dual operation to the $s$-gradient $d_s$ the $s$-divergence $\div_s$. It maps a vector field $F \in \mathcal{M}(\Ep^1_{od}\R^n)$ into a function $\div_s F \in \mathcal{M}(\R^n)$. Its distributional definition is
\begin{equation}\label{eq:divs}
 \div_s(F)[\varphi] = \int_{\R^n} \int_{\R^n} F(x,y)\ d_s \varphi(x,y)\ \frac{dx\ dy}{|x-y|^n}, \quad \varphi \in C_c^\infty(\R^n).
\end{equation}
In particular, we say that a vector field $F$ is divergence free, $\div_s F = 0$, if
\begin{equation}\label{eq:vanishing}
\int_{\R^n} \int_{\R^n} F(x,y)\ \frac{\varphi(x) -\varphi(y) }{|x-y|^s} \frac{dx\ dy}{|x-y|^n} = 0 \quad \mbox{for any $\varphi \in C_c^\infty(\R^n)$}.
\end{equation}
Moreover, the canonical relation to the fractional Laplacian holds true: $\div_s d_s = -(-\lap)^s$, in the sense that
\[
 \int_{\R^n} d_s f \cdot d_sg(x)\ dx = \int_{\R^n} (-\lap)^s f(x)\ g(x)\ dx.
\]
Here, $(-\lap)^s$ stands for 
\[
 (-\lap)^s f(x)  = P.V. \int_{\R^n} \frac{f(x)-f(y)}{|x-y|^{2s}}\ \frac{dy}{|x-y|^n},
\]
or, equivalently,
\[
 \mathcal{F} \brac{(-\lap)^s f}(\xi) = c\, |\xi|^{2s} \mathcal{F} (f)(\xi),
\]
where $\mathcal{F}$ denotes the Fourier transform and $c$ is a multiplicative constant.

For $p \in (1,\infty)$ the natural $L^p$-space  on vector fields $F$, is induced by the norm 
\[
 \|F\|_{L^p(\Ep^1_{od}\R^n)}:= \brac{\int_{\R^n}\int_{\R^n}|F(x,y)|^p \frac{dx\ dy}{|x-y|^n}}^\frac 1p \equiv \big \|\|F\|_{p}(\cdot) \big \|_{L^p(\R^n)}.
\]
Observe that, in particular, we have
\[
 [g]_{W^{s,p}(\R^n)} = \|d_s g\|_{L^p(\Ep^1_{od} \R^n)}.
\]
Also, if $F \in L^p(\Ep^1_{od}\R^n))$ then $F(x,x) = 0$ for almost every $x$.

Lastly, for $\Omega\subseteq \R^n$ we denote
\[
 \|F\|_{L^p(\Ep^1_{od}\Omega)}:= \brac{\iint_{\brac{\Omega\times\R^n}\cup \brac{\R^n\times\Omega}}|F(x,y)|^p \frac{dx\ dy}{|x-y|^n}}^\frac 1p.
\]

A fractional div-curl quantity is then the product $F \cdot d_s g$ where $F$ is $s$-divergence free. For such expressions, we have the fractional counterpart of Theorem~\ref{th:clms}.
\begin{theorem}[div-curl quantities and Hardy space]\label{th:clmstypestimate}
Let $s \in (0,1)$, $p \in (1,\infty)$. For $F \in L^p(\Ep^1_{od} \R^n)$ and $g \in W^{s,p'}(\R^n)$ assume that $\div_s F = 0$. Then $F \cdot d_s g$ belongs to the Hardy space $\mathcal{H}^1(\R^n)$ and we have the estimate
\[
 \|F \cdot d_s g \|_{\mathcal{H}^1(\R^n)} \leq C\  \|F\|_{L^p(\Ep^1_{od} \R^n)}\ \|d_s g\|_{L^{p'}(\Ep^1_{od} \R^n)},
\]
or, equivalently, for any $\varphi \in C_c^\infty(\R^n)$,
\[
 \int_{\R^n} \varphi\ F \cdot d_s g\ dx \leq C\ [\varphi]_{BMO(\R^n)}\ \|F\|_{L^p(\Ep^1_{od} \R^n)}\ \|d_s g\|_{L^{p'}(\Ep^1_{od} \R^n)}.
\]
Here, $C$ is a uniform constant depending on $s$, $p$, and the dimension $n$.
\end{theorem}
The definition of $BMO$ and the Hardy space $\mathcal{H}^1$ can be found in Section~\ref{s:hardyspaceest}
\begin{remark}
It would be interesting to see if the estimate from Theorem~\ref{th:clmstypestimate} could also be proved from the harmonic extension to the upper half space, as is possible for classical div-curl structures and many commutators, see \cite{Lenzmann-Schikorra-2016}. 
\end{remark}

As a first immediate application let us state the following fractional version of Wente's lemma, see \cite{Wente69,BrC84,Tartar84}.

\begin{corollary}[Fractional Wente Lemma]\label{co:wentetypestimate}
Let $s \in (0,1)$, $p \in (1,\infty)$. For $F \in L^p(\Ep^1_{od} \R^n)$ and $g \in W^{s,p'}(\R^n)$ assume that $\div_s F = 0$. Moreover, let $T[\varphi]$ be a linear operator such that for some $\Lambda > 0$,
\[
 T[\varphi] \leq \Lambda \|(-\lap)^{\frac{n}{4}} \varphi \|_{L^{(2,\infty)}(\R^n)} \quad \mbox{for all $\varphi \in C_c^\infty(\R^n)$,}
\]
where $L^{(2,\infty)}$ denotes the weak $L^2$-space.

Then any distributional solution $u \in \dot{W}^{\frac{n}{2},2}(\R^n)$ to
\[
 \laps{n} u = F \cdot d_s g + T\quad \mbox{in $\R^n$},
\]
is continuous and if $\lim_{|x|\to\infty} u(x) = 0$, then
\[
 \|u\|_{L^\infty(\R^n)} \leq C\ \brac{\|F\|_{L^p(\Ep^1_{od} \R^n)}\ \|d_s g\|_{L^{p'}(\Ep^1_{od} \R^n)} + \Lambda},
\]
for a uniform constant $C > 0$ depending only on $s$, $p$, and the dimension $n$.
\end{corollary}
For weak $L^p$-spaces $L^{(p,\infty)}$ and more generally Lorentz spaces $L^{(p,q)}$ we refer, e.g., to \cite{GrafakosCF,Tartar-2007}.
\begin{proof}
This follows from a standard argument, we only give a sketch of the proof: by Theorem~\ref{th:clmstypestimate}, for any $\varphi \in C_c^\infty(\R^n)$,
\[
 \int_{\R^n} (-\lap)^{\frac{n}{4}} u\ (-\lap)^{\frac{n}{4}} \varphi \aleq [\varphi]_{BMO}\, \|F\|_{L^p(\Ep^1_{od} \R^n)}\, \|d_s g\|_{L^{p'}(\Ep^1_{od} \R^n)}+ \Lambda\ \|(-\lap)^{\frac{n}{4}} \varphi\|_{L^{(2,\infty)}(\R^n)}.
\]
Since $[\varphi]_{BMO} \aleq \|(-\lap)^{\frac{n}{4}} \varphi\|_{L^{(2,\infty)}(\R^n)}$, we obtain $(-\lap)^{\frac{n}{4}} u \in L^{(2,1)}_{loc}(\R^n)$. This implies that $u$ is continuous by Sobolev embedding. 
\end{proof}

It is also beneficial to have a localized version of Theorem~\ref{th:clms}. A classical version of this result can be found, e.g., in \cite[Corollary 3]{Strzelecki-1994}. 
\begin{proposition}[Localized div-curl estimate]\label{pr:localclms}
Let $F \in L^p(\Ep^1_{od} \R^n)$ be such that $\div_s F = 0$ and let $g \in W^{s,p'}(\R^n)$ for some $s \in (0,1)$, $p \in (1,\infty)$.

Then, for any ball $B(x_0,r) \subset \R^n$ and any $\varphi \in C_c^\infty(B(x_0,r))$, for a uniform constant $\Lambda > 0$ we have 
\[
 \int_{\R^n} \varphi\ F \cdot d_s g\ dx \leq C\, \brac{[\varphi]_{BMO} + r^{-n}\|\varphi\|_{L^1(\R^n)}}\ \|F\|_{L^p(\Ep^1_{od} B(x_0, \Lambda r))} \ \|d_s g\|_{L^{p'}(\Ep^1_{od} B(x_0,\Lambda r))}.
\]
\end{proposition}

The proofs of Theorem~\ref{th:clmstypestimate} and Proposition~\ref{pr:localclms} can be found in Section~\ref{s:hardyspaceest}.

\section{Fractional   div-curl quantities and half-harmonic maps into spheres}\label{s:halfharmsphere}
As a first application, let us observe how Theorem~\ref{th:clmstypestimate} gives a new, streamlined proof of the continuity of half-harmonic maps into spheres $\S^{N-1}$. Let $\dot{H}^{\frac{1}{2}}(\R)\equiv \dot{W}^{\frac{1}{2},2}(\R)$ be the homogeneous Sobolev space of order $\frac{1}{2}$. By $\dot{H}^{\frac{1}{2}}(\R,\S^{N-1})$ we denote the space of maps $u \in \dot{H}^{\frac{1}{2}}(\R,\R^N)$ such that $u(x) \in \S^{N-1}$ for almost every $x \in \R$.

Half-harmonic maps are solutions $u \in \dot{H}^{\frac{1}{2}}(\R,\S^{N-1})$ to
\[
 \laph u \perp T_u \S^{N-1}.
\]
Equivalently, see \cite{Millot-Sire-2015}, they satisfy
\begin{equation}\label{eq:halfharmeq}
 \laph u^i = u^i |d_{\frac{1}{2}} u|_2^2 \quad \mbox{in $\R$, for $i = 1,\ldots,N$}.
\end{equation}

Our first observation is a conservation law, analogous to Shatah's \eqref{eq:heldiv0}, see \cite{Shatah-1988}. 
\begin{lemma}\label{la:conservationlaw}
A map $u \in \dot{H}^{\frac{1}{2}}(\R,\S^{N-1})$ is a solution to \eqref{eq:halfharmeq} if and only if
\begin{equation}\label{eq:sphereomegadef}
 \Omega_{ik}(x,y) := u^i(x)\, d_{\frac{1}{2}} u^k(x,y) - u^k(x)\, d_{\frac{1}{2}} u^i(x,y)
\end{equation}
satisfies
\begin{equation}\label{eq:ourheldiv0}
 \div_{\frac{1}{2}} \Omega_{ik}  = 0 \quad \mbox{for all }i,k \in \{1,\ldots,N\}.
\end{equation}
\end{lemma}
Let us remark that in \cite{DaLio-Riviere-CAG} Da Lio and Rivi\`{e}re obtained an almost-conservation law for \emph{horizontal} fractional harmonic maps. As a consequence of Lemma~\ref{la:conservationlaw} above we obtain a new proof of
\begin{theorem}\label{th:halfharmreg}
Half-harmonic maps, that is solutions  $u \in \dot{H}^{\frac{1}{2}}(\R,\S^{N-1})$ to \eqref{eq:halfharmeq} are continuous.
\end{theorem}
Regularity for half-harmonic maps was first proved in the pioneering work by Da Lio--Rivi\`{e}re \cite{DaLio-Riviere-1Dsphere}. Another approach was given by Millot--Sire \cite{Millot-Sire-2015} who interpreted the half-harmonic map equation \eqref{eq:halfharmeq} as the free boundary condition of a harmonic map $U: \R^{2}_+ \to \R^N$
\[
 \begin{cases}
\lap U^i = 0 \quad &\mbox{on $\R^2_+$},\\
U = u \quad &\mbox{on $\R \times \{0\}$},
 \end{cases}
\]
observing that then $\partial_\nu U = c\ \laph u$ on $\R$. Then regularity theory follows from the known regularity results for such free boundary harmonic maps, see \cite{Scheven-2006}.

The proof of Theorem~\ref{th:halfharmreg} that we give here follows very closely the original proof for harmonic maps into spheres by H\'{e}lein \cite{Helein90}.

\begin{proof}[Proof of Theorem~\ref{th:halfharmreg}]As in the local case we rewrite the right-hand side of \eqref{eq:halfharmeq}. Recall that we use Einstein's summation convention. With $\Omega_{ik}$ from \eqref{eq:sphereomegadef} we find
\begin{equation}\label{eq:halfharmrew}
 \laph u^i(x) = u^i(x)\ |d_{\frac{1}{2}} u|_2^2(x) \equiv  u^i(x)\  \langle d_{\frac{1}{2}} u^k , d_{\frac{1}{2}}u^k \rangle (x)=
  \langle \Omega_{ik}, d_{\frac{1}{2}}u^k \rangle(x) + T(x)
\end{equation}
where 
\[
 T(x) := \langle d_{\frac{1}{2}} u^i, u^k(x)\, d_{\frac{1}{2}}u^k \rangle (x).
\]
As for $T$ we have
\begin{equation}\label{eq:ourhell2infty}
 \left |\int_{\R} T(x)\ \varphi(x)\ dx \right | \aleq \|\lapv \varphi\|_{L^{(2,\infty)}(\R)} \quad\text{ for every }\varphi\in C^\infty_c(\R).
\end{equation}
Assuming \eqref{eq:ourhell2infty} and in view of Lemma~\ref{la:conservationlaw} we have found that the equation \eqref{eq:halfharmrew} exhibits a fractional div-curl structure on the right-hand side. Thus, it falls into the realm of the fractional Wente lemma, Corollary~\ref{co:wentetypestimate}. 

Hence, Theorem~\ref{th:halfharmreg} is proven once Lemma~\ref{la:conservationlaw} and \eqref{eq:ourhell2infty} are established.
%
\end{proof}

\subsection{Proof of Lemma~\ref{la:conservationlaw}}
\begin{proof}[Proof of \eqref{eq:halfharmeq} $\Rightarrow$ \eqref{eq:ourheldiv0}]
We compute the fractional divergence  $\div_{\frac{1}{2}} \Omega_{ik}$, see \eqref{eq:divs}. For any $\varphi \in C_c^\infty(\R)$,
\[
\begin{split}
  &\int_{\R}\int_{\R} \Omega_{ik}(x,y)\, d_{\frac{1}{2}} \varphi(x,y)\, \frac{dx\ dy}{|x-y|} \\
  &=\int_{\R}\int_{\R} \brac{u^i(x)\, d_{\frac{1}{2}} \varphi(x,y)\ d_{\frac{1}{2}} u^k(x,y) - u^k(x)\, d_{\frac{1}{2}} \varphi(x,y)\, d_{\frac{1}{2}} u^i(x,y)} \frac{dx\ dy}{|x-y|}. \\
\end{split}
 \]
Now a simple computation confirms the product rule for $d_{\frac{1}{2}}$,
\[
 u^i(x)\, d_{\frac{1}{2}} \varphi(x,y) = d_{\frac{1}{2}} (u^i \varphi)(x,y) - d_{\frac{1}{2}} u^i(x,y)\, \varphi(y).
\]
Thus, 
\[
\begin{split}
  &\int_{\R}\int_{\R} \Omega_{ik}(x,y) d_{\frac{1}{2}} \varphi(x,y)\frac{dx\ dy}{|x-y|} \\
  &=\int_{\R}\int_{\R} \brac{d_{\frac{1}{2}} (u^i\varphi)(x,y)\ d_{\frac{1}{2}} u^k(x,y) - d_{\frac{1}{2}} (u^k\varphi)(x,y)\, d_{\frac{1}{2}} u^i(x,y)} \frac{dx\ dy}{|x-y|} \\
  &\quad -\int_{\R}\int_{\R} \varphi(y)\brac{ d_{\frac{1}{2}} u^i(x,y)\ d_{\frac{1}{2}} u^k(x,y) - d_{\frac{1}{2}} u^k(x,y)\, d_{\frac{1}{2}} u^i(x,y)} \frac{dx\ dy}{|x-y|}. \\
\end{split}
 \]
The last line is zero. With \eqref{eq:halfharmeq} we find
\[
\begin{split}
 &\int_{\R}\int_{\R} \brac{d_{\frac{1}{2}} (u^i\varphi)(x,y)\ d_{\frac{1}{2}} u^k(x,y) - d_{\frac{1}{2}} (u^k\varphi)(x,y)\, d_{\frac{1}{2}} u^i(x,y)} \frac{dx\ dy}{|x-y|}\\
 &=\int_{\R} \brac{u^i(x)\ u^k(x)  - u^k(x)\ u^i(x)  } \varphi(x)\ \|d_{\frac{1}{2}} u\|_2^2(x)\ dx = 0.
\end{split}
\]
\eqref{eq:ourheldiv0} is established.
\end{proof}

\begin{proof}[Proof of \eqref{eq:ourheldiv0} $\Rightarrow$ \eqref{eq:halfharmeq}]
Equation \eqref{eq:ourheldiv0} readily implies 
\[
 0=u^i(x) \laph u^j(x) - u^j(x) \laph u^i(x),
\]
that is
\[
 0 = u \wedge \laph u,
\]
which is equivalent to
\[
 \laph u \perp T_u \S^{N-1}.
\]
Thus, $u$ is a half-harmonic map.
\end{proof}

\subsection{Proof of \eqref{eq:ourhell2infty}}
For any $x,y \in \R$,
\[
 u^k(x)\, d_{\frac{1}{2}}u^k(x,y) = u^k(x)\, \frac{(u^k(x)-u^k(y))}{|x-y|^{\frac{1}{2}}}.
\]
Since $|u(x)| \equiv 1$ and thus $(u(x)-u(y))\cdot(u(x)+u(y)) = 0$, we find
\[
 u^k(x)\, d_{\frac{1}{2}}u^k(x,y) = \frac{1}{2} \frac{|u(x)-u(y)|^2}{|x-y|^{\frac{1}{2}}}.
\]
Thus,
\[
 \int_{\R} \varphi(x)\, u^k(x)\ d_{\frac{1}{2}} u^i\cdot d_{\frac{1}{2}}u^k(x)\ dx =\frac{1}{2} \int_{\R}\int_{\R} \varphi(x)\  \frac{u^i(x)-u^i(y)}{|x-y|^{\frac{1}{2}}}\, \frac{|u(x)-u(y)|^2}{|x-y|^{\frac{1}{2}}} \frac{dy\, dx}{|x-y|}.
 \]
Interchanging $x$ and $y$, we arrive at
\[
\begin{split}
 \int_{\R} \varphi(x)\ u^k(x)\ d_{\frac{1}{2}} u^i\cdot d_{\frac{1}{2}}u^k(x)\ dx
 &\leq \frac{1}{4} \int_{\R} \int_{\R}\frac{|u(x)-u(y)|^3\ |\varphi(x)-\varphi(y)|}{|x-y|^{2}}\ dx\ dy\\
 &=\int_{\R} \int_{\R} |d_{\frac{1}{6}}u(x,y))|^3\ |d_{\frac{1}{2}}\varphi(x,y)|\ \frac{dx\ dy}{|x-y|}.
\end{split}
 \]
From this one obtains by interpolation,
\[
 \int_{\R} \varphi(x)\ u^k(x)\ d_{\frac{1}{2}} u^i\cdot d_{\frac{1}{2}}u^k(x)\ dx \aleq \|((-\lap)^{\frac{1}{6}} u)^3\|_{L^{(2,1)}(\R)}\ \|\lapv \varphi\|_{L^{(2,\infty)}(\R)}.
\]
Moreover, by Sobolev embedding,
\[
\begin{split}
 \|((-\lap)^{\frac{1}{6}} u)^3\|_{L^{(2,1)}(\R)} = \|(-\lap)^{\frac{1}{6}} u\|_{L^{(6,3)}(\R)}^3 &\aleq \|\lapv u\|_{L^{(2,3)}(\R)}^3\\
 &\aleq \|\lapv u\|_{L^2(\R)}^3 = [u]_{W^{\frac{1}{2},2}(\R)}^3.
 \end{split}
\]
This establishes \eqref{eq:ourhell2infty}.
\qed

\section{Fractional   div-curl quantities and systems with nonlocal antisymmetric potential and half-harmonic maps into general manifolds}\label{s:antisym}
Here we study the regularity theory for a nonlocal analogue of \eqref{eq:confoeq}. Let $u \in \dot{H}^{\frac{1}{2}}(\R,\R^N)$ be a solution to
\begin{equation}\label{eq:nlocantisym}
 \laph u^i = \Omega_{ij} \cdot d_{\frac12} u^j,
\end{equation}
for some antisymmetric $\Omega_{ij} = -\Omega_{ji} \in L^2(\Ep^1_{od} \R)$.

Observe that the antisymmetric potential $\Omega$ is not a pointwise function, but rather acts as a nonlocal operator: one could write the equation above as
\[
 \laph u^i = \Omega_{ij}(u^j),
\]
where
\[
 \Omega_{ij}(f)(x) := \int_{\R} \Omega_{ij}(x,y) (f(x)-f(y))\ \frac{dy}{|x-y|^{\frac{3}{2}}}.
\]
%
In \cite{DaLio-Riviere-1Dmfd} Da Lio and Rivi\`{e}re studied the regularizing effects of the equation
\[
 \laph u^i(x) = \Omega_{ij}(x)\, \lapv u^j(x),
\]
where $\Omega_{ij} = -\Omega_{ji} \in L^2(\R)$ is a \emph{function}. In \cite{Schikorra-eps} the second-named author studied the regularity theory for another class of antisymmetric \emph{nonlocal} operators,
\[
 \laph u^i = \Omega_{ij} \mathcal{H}[\lapv u^j],
\]
where $\mathcal{H}$ is the Hilbert transform.

Here, in the spirit of the celebrated work of Rivi\`{e}re \cite{Riviere-2007}, we develop the regularity theory of nonlocal antisymmetric systems of the form \eqref{eq:nlocantisym}. Namely, we have
\begin{theorem}[Regularity for systems with nonlocal antisymmetric operator]\label{th:nlao}
Let $u \in \dot{H}^{\frac{1}{2}}(\R,\R^N)$ be a weak solution to 
\begin{equation}\label{eq:genmfdharmmapeq}
 \laph u^i = \Omega_{ij} \cdot d_{\frac{1}{2}} u^j + R \quad \mbox{in $\R$}.
\end{equation}
Assume that $\Omega_{ij} = - \Omega_{ji} \in L^2(\Ep^1_{od} \R)$ and that $R$ satisfies, for any $\varphi \in C_c^\infty(\R)$,
\begin{equation}\label{eq:genmfd:errorestimate}
 \int_{\R} R \varphi \aleq \int_{\R}|\varphi(x)| \int_{\R} |d_{\frac{1}{3}}u(x,y)|^3\ \frac{dy\ dx}{|x-y|} + \int_{\R}\int_{\R} |d_{\frac{1}{4}}u(x,y)|^2\ |d_{\frac{1}{2}}\varphi(x,y)| \frac{dy\ dx}{|x-y|}.
\end{equation}
Then $u$ is H\"older continuous.
\end{theorem}
We postpone the proof and first mention an application. 
Regularity theory for half-harmonic maps from $\R$ into a smooth, compact manifold $\mathcal{N} \subset \R^N$ without boundary follows from Theorem~\ref{th:nlao}. Indeed, 
just as the harmonic map equation \eqref{eq:harmmapeqperp} can be brought into the form \eqref{eq:confoeq}, the half-harmonic map equation 
\begin{equation}\label{eq:halfharmperp}
 \laph u \perp T_u \mathcal{N} \quad \mbox{in $\R$}
\end{equation}
can be brought into the form \eqref{eq:genmfdharmmapeq}.
\begin{proposition}\label{pr:genmfdEL}
Let $u \in \dot{H}^{\frac{1}{2}}(\R,\mathcal{N})$ be a half-harmonic map into a general smooth manifold without boundary $\mathcal{N}\subset \R^N$, i.e., a distributional solution to \eqref{eq:halfharmperp}. Then $u$ solves \eqref{eq:genmfdharmmapeq} for some $\Omega_{ij} = - \Omega_{ji} \in L^2(\Ep^1_{od} \R)$ and some $R$ which satisfies \eqref{eq:genmfd:errorestimate}.
\end{proposition}

The proof of Proposition~\ref{pr:genmfdEL}, which we give in Section~\ref{s:genmfdEL}, follows essentially the local argument used to obtain \eqref{eq:confoeq}. The only difference is that while $d u$ is tangential, $d_su(x,y)$ is not --- this is why the error term
$R$ appears. But since $d_s u(x,y)$ is tangential up to a quadratic error, see Lemma~\ref{la:uxmuytangential}, $R$ is benign.

Thus, as a corollary of Theorem~\ref{th:nlao} and Proposition~\ref{pr:genmfdEL}, we obtain

\begin{theorem}[Da Lio, Rivi\`{e}re \cite{DaLio-Riviere-1Dmfd}]
Half-harmonic maps from the line $\R$ into a general manifold $\mathcal{N}$ are H\"older continuous and in fact smooth. 
\end{theorem}
It suffices to show the H\"older continuity. Higher regularity follows from bootstrapping and the growth of the right-hand side -- the antisymmetry and the precise right-hand side structure are not relevant. See \cite{Schikorra-eps}.

\subsection{Proof of Theorem~\ref{th:nlao}}
As in the local case \eqref{eq:gaugelocal}, the first step is to find a good gauge $P \in \dot{H}^{\frac{1}{2}}(\R,SO(N))$.
\begin{theorem}\label{th:gauge}
For $\Omega_{ij} = -\Omega_{ji} \in L^2(\Ep^1_{od}\R)$ there exists $P \in \dot{H}^{\frac{1}{2}}(\R,SO(N))$ such that
\[
 \div_{\frac{1}{2}} \Omega^P_{ij} = 0 \quad \mbox{for all }i,j \in \{1,\ldots,N\},
\]
where
\[
 \Omega^P = \frac{1}{2} \brac{d_{\frac{1}{2}}P(x,y) \brac{P^T(y) + P^T(x)} -  P(x) \Omega(x,y) P^T(y) - P(y) \Omega(x,y) P^T(x)}. 
\]
\end{theorem}
This choice of gauge \cite{Uhlenbeck-1982}, or moving frame \cite{Helein91}, is obtained from a minimization argument, cf. \cite{Schikorra-2010}, and is postponed to Section~\ref{s:gaugeconstruction}.

Having this good choice for $P$, we rewrite the equation.
\begin{lemma}\label{la:gettingomegap}
Let $u$ be a solution to \eqref{eq:nlocantisym}. For $P$ as in Theorem~\ref{th:gauge} and any $\varphi \in C_c^\infty(\R)$ we have,
\[
\begin{split}
\div_{\frac{1}{2}} (P_{ik} d_{\frac{1}{2}} u^k) =\Omega_{ij}^P \cdot d_{\frac{1}{2}}u^\ell\ P_{j\ell} + R.
\end{split}
  \]
Here, for some $R_1 \in L^2(\Ep^1_{od}\R)$, some $g \in \dot{H}^{\frac{1}{2}}(\R)$, and for any $\varphi \in C_c^\infty(\R)$,
\[
\begin{split}
\left |\int_{\R} R \varphi \right |&\aleq \int_{\R}\int_{\R} |R_1(x,y)|\ |d_{\frac{1}{4}} g(x,y)| |d_{\frac{1}{4}}u(x,y)|\ |\varphi(x)| \frac{dx\ dy}{|x-y|}.
\end{split}
\]
\end{lemma}

\begin{proof}
We have
\[
\begin{split}
 &\int_{\R}\int_{\R} P_{ik}(x) \brac{u^k(x)-u^k(y)}\ \brac{\varphi(x) - \varphi(y)} \frac{dx\ dy}{|x-y|^{2}}\\
 &=\int_{\R}\int_{\R} \brac{u^k(x)-u^k(y)}\ \brac{P_{ik}(x) \varphi(x) - P_{ik}(y) \varphi(y)} \frac{dx\ dy}{|x-y|^{2}}\\
 &\quad+\int_{\R}\int_{\R} \brac{u^k(x)-u^k(y)}\ \brac{P_{ik}(y)  - P_{ik}(x)}\ \varphi(y) \frac{dx\ dy}{|x-y|^{2}}.
 \end{split}
 \]
By \eqref{eq:nlocantisym} and since $P^T_{\ell j}(x)\ P_{jm}(x) = \delta_{\ell m}$,
\[
\begin{split}
 &\int_{\R}\int_{\R} \brac{u^k(x)-u^k(y)}\ \brac{P_{ik}(x) \varphi(x) - P_{ik}(y) \varphi(y)} \frac{dx\ dy}{|x-y|^{2}}\\
 &=\int_{\R}\int_{\R} \brac{P_{ik}(x)\ \Omega_{k\ell}(x,y)\ P^T_{\ell j}(x)}\ P_{j{m}}(x) d_{\frac{1}{2}}u^{m}(x,y)\ \varphi(x)\frac{dx\ dy}{|x-y|}.
\end{split}
 \]
Also, interchanging $x$ and $y$ and using again $P^T_{\ell j}(x)\ P_{jm}(x) = \delta_{\ell m}$,
\[
\begin{split}
  &\int_{\R}\int_{\R} \brac{u^k(x)-u^k(y)}\ \brac{P_{ik}(y)  - P_{ik}(x)}\ \varphi(y) \frac{dx\ dy}{|x-y|^{2}}\\
  &=\int_{\R}\int_{\R} \brac{u^k(x)-u^k(y)}\ \brac{P_{ik}(y)  - P_{ik}(x)}\ \varphi(x) \frac{dx\ dy}{|x-y|^{2}}\\
  &=-\int_{\R}\int_{\R} d_{\frac{1}{2}}P_{i\ell}(x,y)\ P_{\ell j}^{{T}}(x)\ P_{j{m}}(x)\, d_{\frac{1}{2}}u^{m}(x,y)\ \varphi(x)\frac{dx\ dy}{|x-y|}\\
\end{split}
\]
Thus,
 \[
 \begin{split}
 &\int_{\R}\int_{\R} P_{ik}(x) \brac{u^k(x)-u^k(y)}\ \brac{\varphi(x) - \varphi(y)} \frac{dx\ dy}{|x-y|^{2}}\\
 &={-}\int_{\R}\int_{\R} \Omega_{ij}^P(x,y)\ P_{j\ell}(x)\ d_{\frac{1}{2}}u^\ell(x,y)\ \varphi(x) \frac{dx\ dy}{|x-y|} +R,
\end{split}
 \]
where
\[
 |R| \aleq \int_{\R}\int_{\R} \brac{|\Omega(x,y)| + |d_{\frac{1}{2}} P(x,y)|} |P(x)-P(y)|\ |d_{\frac{1}{2}}u(x,y)|\ |\varphi(x)| \frac{dx\ dy}{|x-y|}.
\]
\end{proof}
From Lemma~\ref{la:gettingomegap} and Theorem~\ref{th:gauge} we see that we have found a div-curl quantity. Now we can apply the localized div-curl estimate, Proposition~\ref{pr:localclms}. With this, one obtains a decay estimate following the typical procedure for critical nonlocal equations, see \cite{DaLio-Riviere-1Dmfd,DaLio-2013,Schikorra-eps}. We only give the main ideas of the proof. For any ball $B(x_0,r) \subset \R$ and any test-function $\varphi \in C_c^\infty(B(x_0,r))$
\[
 \int_{\R} P_{ik} d_{\frac{1}{2}} u^k\ d_{\frac{1}{2}} \varphi = \int_{\R} \Omega_{ij}^P \cdot d_{\frac{1}{2}}u^\ell\ P_{j\ell}\ \varphi + \int_{\R} R \varphi
\]
For any suitably large $\Lambda > 2$,
\[
 \tilde{u} := \eta_{B(x_0,\Lambda r)} (u-(u)_{B(x_0,\Lambda r)}).
\]
for a cutoff-function $\eta \in C_c^\infty(B(0,1))$ and $\eta_{B(x_0,\rho)}(x) := \eta((x-x_0)/\rho)$. Then
\[
\int_{\R} \Omega_{ij}^P \cdot d_{\frac{1}{2}}u^\ell\ P_{j\ell}\ \varphi =\int_{\R} \Omega_{ij}^P \cdot d_{\frac{1}{2}}\tilde{u}^\ell\ P_{j\ell}\ \varphi + \int_{\R} \Omega_{ij}^P \cdot d_{\frac{1}{2}} \brac{u^\ell-\tilde{u}^\ell}\ P_{j\ell}\ \varphi.
\]
By the disjoint support of $\varphi$ and $u-\tilde{u}$, one can show, for some $\sigma >0$ which will change from line to line,
\[
 \int_{\R} \Omega_{ij}^P \cdot d_{\frac{1}{2}} \brac{u^\ell-\tilde{u}^\ell}\ P_{j\ell} \aleq \sum_{k=1}^\infty (2^k \Lambda)^{-\sigma} \|\lapv u\|_{L^{(2,\infty)}(B(x_0,2^k \Lambda r)}.
\]
Since $\Omega^P_{ij}$ is divergence free,
\[
 \int_{\R} \Omega_{ij}^P \cdot d_{\frac{1}{2}}\tilde{u}^\ell\ P_{j\ell}\ \varphi = -\int_{\R} \Omega_{ij}^P \cdot d_{\frac{1}{2}}\brac{P_{j\ell}\ \varphi}\ \tilde{u}^\ell.
\]
Now we apply Proposition~\ref{pr:localclms},
\[
\begin{split}
  \int_{\R} \Omega_{ij}^P &\cdot d_{\frac{1}{2}}\tilde{u}^\ell\ P_{j\ell}\ \varphi\\
  &\aleq \brac{[\tilde{u}]_{BMO} + (\Lambda r)^{-1} \|\tilde{u}\|_{L^1(\R)}}\ \|\Omega^P\|_{L^2(\Ep^1_{od} B(x_0,\Lambda^2 r))}\ \| d_{\frac{1}{2}}\brac{P \varphi}\|_{L^2(\Ep^1_{od} B(x_0,\Lambda^2 r))}.
\end{split}
  \]
By Sobolev embedding, $\|f\|_{BMO} \aleq \|\lapv f\|_{L^{(2,\infty)}(\R)}$, which suitably localized gives
\[
\begin{split}
 [\tilde{u}]_{BMO} + (\Lambda r)^{-1} \|\tilde{u}\|_{L^1}& \aleq [\tilde{u}]_{BMO} \\
 &\aleq \|\lapv u\|_{L^{(2,\infty)}(B(\Lambda^2 r))}\ + \sum_{k=1}^\infty (2^k \Lambda)^{-\sigma} \|\lapv u\|_{L^{(2,\infty)}(B(x_0,2^k \Lambda^2 r)}. 
\end{split}
 \]
Lastly,
\[
 \|d_{\frac{1}{2}}\brac{P \varphi}\|_{L^2(\Ep^1_{od} B(x_0,\Lambda^2 r))} \aleq \brac{\|P\|_{L^\infty(\R)} + \|\lapv P\|_{L^2(\R)}}\, \brac{\|\lapv \varphi\|_{L^2(\R)}\ +  \|\varphi\|_{L^\infty(\R)}}. 
\]
In conclusion,
\[
\begin{split}
\bigg |\int_{\R} P_{ik} &d_{\frac{1}{2}} u^k\ d_{\frac{1}{2}} \varphi \bigg | \\
&\aleq \lambda\,  (\|\varphi\|_{L^\infty(\R)} + \|\lapv \varphi\|_{L^2(\R)})\ 
\brac{[u]_{BMO(B(x_0,\Lambda^2 r))} + \|\lapv u\|_{L^{(2,\infty)}(B(x_0,\Lambda^2 r))}}\\
&\quad+ \sum_{k=1}^\infty (2^k \Lambda)^{-\sigma} \|\lapv u\|_{L^{(2,\infty)}(B(x_0,2^k \Lambda^2 r)},
\end{split}
  \]
where
\[
 \lambda = \|\Omega^P\|_{L^2(\Ep^1_{od} B(x_0,\Lambda^2 r))}
\]
can be chosen arbitrarily small if we restrict our attention to small enough balls, $r \ll 1$.

By a duality argument, see \cite[Lemma 5.18]{Schikorra-eps}, we find some $\varphi \in C_c^\infty(B(x_0,\Lambda r))$, $\|\varphi \|_{\infty} + \|\lapv \varphi\|_{L^{(2,1)}(\R)} \leq 1$ such that 
\[
\begin{split}
 \|\lapv u\|_{L^{(2,\infty)}(B(x_0,r))} &\aleq \left |\int_{\R} P_{ik} \lapv u^k\ \lapv \varphi \right | + \Lambda^{-\sigma} \|\lapv u\|_{L^{(2,\infty)}(B(x_0,\Lambda^2 r))} \\
 & \quad + \sum_{k=1}^\infty (2^k \Lambda)^{-\sigma} \|\lapv u\|_{L^{(2,\infty)}(B(x_0,2^k \Lambda^2 r)}.
 \end{split}
\]
Now, we have the following commutator-like estimate.
\begin{lemma}
For any $\varphi \in C_c^\infty(\R)$
\[
\begin{split}
 \bigg |\int_{\R} P\ d_{\frac{1}{2}} u\cdot d_{\frac{1}{2}} \varphi -& \int_{\R} P\lapv u\ \lapv \varphi \bigg |\\  &\aleq \|\lapv \varphi\|_{L^2(\R)}\ \bigg(\|\lapv P\|_{L^{2}(B(x_0,\Lambda r))} \|\lapv u\|_{L^{(2,\infty)}(B(x_0,\Lambda r))} \\
 &\quad+ \sum_{k=1}^\infty (2^k \Lambda)^{-\sigma} \|\lapv u\|_{L^{(2,\infty)}(B(x_0,2^k \Lambda r))}\bigg).
 \end{split}
\]
\end{lemma}
\begin{proof}
We use the formula\[
        \varphi(z) = c\int_{\R} |x-z|^{\frac{1}{2}} \lapv \varphi(z)\ dz
       \]
and have
 \[
\begin{split}
 &\int_{\R} \int_{\R} \frac{P(x)(u(x)-u(y))(\varphi(x)-\varphi(y))}{|x-y|^2}\ dx\ dy - \int_{\R} P(z)\lapv u(z)\ \lapv \varphi(z)\ dz\\
&= \int_{\R} \int_{\R} \int_{\R} \frac{(P(x)-P(z))(u(x)-u(y))  \brac{|x-z|^{-\frac12}-|y-z|^{-\frac12}}}{|x-y|^2}\ dx\ dy\ \lapv \varphi(z)\ dz.
\end{split}
 \]
Now an analysis similar to that of \cite[Lemma 6.5, Lemma 6.6.]{Schikorra-CPDE} completes the proof. 
\end{proof}

We conclude that, on every ball $B(x_0,r) \subset \R$,
\[
\begin{split}
 \|\lapv u\|_{L^{(2,\infty)}(B(x_0,r))} &\aleq (\lambda + \Lambda^{-\sigma}) \|\lapv u\|_{L^{(2,\infty)}(B(x_0,\Lambda^2 r))} \\
 &\quad + \sum_{k=1}^\infty (2^k \Lambda)^{-\sigma} \|\lapv u\|_{L^{(2,\infty)}(B(x_0,2^k \Lambda^2 r)}.
 \end{split}
\]
This can be seen as a good decay estimate: For any $\eps > 0$ we find a large $\Lambda$ such that for all small enough radii $r \ll 1$ (so that $\lambda$ is small),
\[
\begin{split}
 \|\lapv u\|_{L^{(2,\infty)}(B(x_0,r))} &\leq \eps\, \|\lapv u\|_{L^{(2,\infty)}(B(x_0,\Lambda^2 r))} \\
 &\quad + \eps\, \sum_{k=1}^\infty (2^k \Lambda)^{-\sigma} \|\lapv u\|_{L^{(2,\infty)}(B(x_0,2^k \Lambda^2 r)}.
 \end{split}
\]
Now an iteration argument, see \cite[Lemma A.8]{Blatt-Reiter-Schikorra-2016}, implies that there is a $\theta > 0$ such that
\[
 \|\lapv u\|_{L^{(2,\infty)}(B(x_0,r))} \aleq r^\theta \quad \mbox{for all }r > 0.
\]
H\"older continuity of $u$ then follows from Sobolev embedding on Morrey spaces, see \cite{Adams-1975}. This proves Theorem~\ref{th:nlao}. \qed

\section{Fractional div-curl quantities and \texorpdfstring{$W^{s,p}$}{W(s,p)}-harmonic maps into homogeneous manifolds}\label{s:homo}
For $s \in (0,1)$, $p \in (1,\infty)$ $W^{s,p}$-harmonic maps from $\R^n$ into a manifold $\mathcal{N} \subset \R^N$ are solutions $u \in \dot{W}^{s,p}(\R^n,\mathcal{N})$ to
\[
 \div_s (|d_s u|^{p-2} d_s u^i) \perp T_u \mathcal{N}.
\]
These are exactly the critical points of the energy
\[
 \mathcal{E}_{s,p}(u) := \int_{\R^n}\int_{\R^n} \brac{\frac{|u(x)-u(y)|}{|x-y|^s}}^p\ \frac{dx\ dy}{|x-y|^{n}} \equiv \|d_s u\|_{L^p(\R^n)}^p.
\]
The operator $\div_s (|d_s u|^{p-2} d_s u^i)$ is often referred to as fractional $p$-Laplacian $(-\lap)^s_p$, whose regularity theory has received a lot of attention lately, see, e.g., \cite{DiCastro-Kuusi-Palatucci-2016,Kuusi-Mingione-Sire-2015,Brasco-Lindgren-2017,Schikorra-2016}.

The strategy for half-harmonic maps into spheres from Section~\ref{s:halfharmsphere}, that is $W^{\frac{1}{2},2}$-harmonic maps can be extended to $W^{s,\frac{n}{s}}$-harmonic maps, into round targets. Namely we obtain H\"older regularity for $W^{s,\frac{n}{s}}$-harmonic maps from $n$-dimensional sets into homogeneous spaces. 
The argument now follows the corresponding classical arguments of Strzelecki \cite{Strzelecki-1994} and Toro--Wang \cite{Toro-Wang} for $n$-harmonic maps into spheres and homogeneous manifolds, respectively.

First, we rewrite the equation.
\begin{lemma}[Euler--Lagrange Equations]\label{la:homoEL}
Let $s \in (0,1)$ and $p \in (1,\infty)$. For any $u \in W^{s,p}(\R^n,\mathcal{N})$ which is a $W^{s,p}$-harmonic map into a homogeneous Riemannian manifold $\mathcal{N}$ equivariantly embedded into $\R^N$. Then, 
\begin{equation}\label{eq:homoel}
 \div_s (|d_s u|^{p-2} d_s u^i)= \sum_{\alpha=1}^m |d_s u|^{p-2}\ \Omega_\alpha \cdot d_s Y_\alpha^i(u) + R.
\end{equation}
Here, $\{Y_\alpha\}_{\alpha=1}^m: \mathcal{N} \to \R^N$ is a family of $m$ smooth tangent vector fields on $\n$, $\Omega_\alpha \in L^p(\Ep^1_{od} \R^n)$ and satisfies
\begin{equation}\label{eq:homodiv0}
 \div_s \brac{|d_s u|^{p-2}\Omega_\alpha} = 0 \quad \mbox{in $\R^n$, for $\alpha =1\ldots,m$}.
\end{equation}
The error term $R$ satisfies,
\[
\left |\int_{\R^n} R\ \varphi \right | \aleq  \int_{\R^n}\int_{\R^n} |d_{\frac{s}{p'}}u(x,y)|^p\ |d_{s}\varphi(x,y)|\ \frac{dy\, dx}{|x-y|^n}.
\]
\end{lemma}
The proof is a direct fractional analogue of the arguments in \cite{Helein91-sym, Strzelecki-1994, Toro-Wang}, we postpone it to Section~\ref{s:homoEL}.

From Lemma~\ref{la:homoEL} and Theorem~\ref{th:clmstypestimate} we obtain the following regularity theorem. This generalizes the second-named author's regularity result for $W^{s,\frac{n}{s}}$-harmonic maps into spheres \cite{Schikorra-CPDE} to homogeneous target manifolds. Let us stress that even for round spheres our argument is much simpler.
\begin{theorem}\label{th:liegroupreg}
Let $s \in (0,1)$, $p = \frac{n}{s} \in [2,\infty)$ and $u \in W^{s,p}(\R^n,\mathcal{N})$ be a $W^{s,p}$-harmonic map, where $\mathcal{N}$ is a homogeneous Riemannian manifold equivariantly embedded into $\R^N$. Then $u$ is H\"older continuous.
\end{theorem}

\begin{proof}[Sketch of the proof]
Let $B(x_0,r) \subset \R^n$ be a ball centered in $x_0$ and $r > 0$.  For a typical cutoff function $\eta \in C_c^\infty(B(0,2))$, $\eta \equiv 1$ on $B(0,1)$, let $\eta_{B(x_0,r)}(x) = \eta((x-x_0)/r)$. Define the test-function $\varphi$ by
\[
 \varphi := \eta_{B(x_0,r)} (u-(u)_{B(x_0,r)}),
\]
where $(u)_{B(x_0,r)}$ denotes the mean value of $u$ on $B(x_0,r)$.
Observe that for any $\Lambda > 4$, we have the following estimates
\[[\vp]_{W^{s,p}(\R^n)}\aleq \|d_s u\|_{L^{p}(\Ep^1_{od}B(x_0, \Lambda r))}\] 
and since $p=\frac{n}{s}$, 
\[
[\varphi]_{BMO} + r^{-n} \|\varphi\|_{L^1(\R^n)} \aleq \|d_s u\|_{L^{p}(\Ep^1_{od}B(x_0, \Lambda r))}.
\]
We also denote by
\[
 [u]^{p}_{W^{s,p}(B(x_0,r))} := \int_{B(x_0,r)}\ \int_{B(x_0,r)} \frac{|u(x)-u(y)|^{p}}{|x-y|^{n+sp}}\ dx\ dy.
\]
By the usual cutoff-arguments, see for example \cite[Lemma 4.1]{Schikorra-CPDE} we find for any $\eps > 0$ a constant $C_\eps > 0$ such that
\[
 \begin{split}
 [u]^{p}_{W^{s,p}(B(x_0,r))} &\leq  \int_{\R^n} |d_s u|^{p-2} d_s u \cdot d_s \varphi(y)\ dy\\
&\quad +\eps\ [u]^{p}_{W^{s,p}(B(x_0,\Lambda r))} + C_\eps \brac{[u]^{p}_{W^{s,p}(B(x_0,\Lambda r))} - [u]^{p}_{W^{s,p}(B(x_0,r))}}\\
  &\quad +\sum_{k=1}^\infty (2^{k} \Lambda)^{-\sigma} [u]_{W^{s,p}(B(x_0,2^k \Lambda^2 r))}^p.
\end{split}
 \]
Here $\sigma > 0$ is a constant that may vary from line to line.

For the first term, we use the equation \eqref{eq:homoel}
\[
\int_{\R^n} |d_s u|^{p-2} d_s u \cdot d_s \varphi = \int_{\R^n} |d_s u|^{p-2}\ \Omega_\alpha \cdot d_s Y_\alpha^i(u)\ \varphi + \int_{\R^n} R\varphi.
\]
For the error term $R$ the usual cutoff arguments and Sobolev inequality lead to
\[
 \int_{\R^n} R\varphi \aleq \|d_s u\|_{L^p(\Ep^1_{od} B(x_0,\Lambda r))}^{p}\  \|d_s\varphi\|_{L^p(\Ep^1_{od} B(x_0,\Lambda r))}.
\]
Moreover, we use the div-curl structure induced by \eqref{eq:homodiv0} and Proposition~\ref{pr:localclms}, to obtain
\[\begin{split}
  \int_{\R^n} &|d_s u|^{p-2} \Omega_\alpha \cdot d_s Y_\alpha^i(u)\ \varphi\\
  &\aleq \||d_s u|^{p-2} \Omega\|_{L^{p'}(\Ep^1_{od} B(x_0,\Lambda r))}\ \|d_s u\|_{L^{p}(\Ep^1_{od}B(x_0,\Lambda r))}\ \brac{[\varphi]_{BMO} + r^{-n} \|\varphi\|_{L^1(\R^n)}}.
  \end{split}
\]
Thus, with the estimate for $\varphi$,
\[
 \int_{\R^n} |d_s u|^{p-2} d_s u \cdot d_s \varphi \aleq \brac{\|\Omega\|_{L^p(\Ep^1_{od} B(\Lambda r))}+ \|d_s u\|_{L^p(\Ep^1_{od} B(\Lambda r))}}\ \|d_s u\|_{L^p(\Ep^1_{od} B(\Lambda r))}^{p}.
\]
For any $\eps > 0$, by absolute continuity of integrals, we find $R > 0$ small enough so that for any $r \in (0,\Lambda^{-1} R)$ we have
\[
 \brac{\|\Omega\|_{L^p(\Ep^1_{od} B(x_0, \Lambda r))}+ \|d_s u\|_{L^p(\Ep^1_{od} B(x_0, \Lambda r))}} < \eps.
\]
Also,
\[
 \|d_s u\|_{L^p(\Ep^1_{od} B(x_0, \Lambda r))}^{p} \aleq [u]_{W^{s,p}(B(x_0, \Lambda^2 r))}^p + \sum_{k=1}^\infty \brac{2^{k} \Lambda}^{-\sigma} [u]_{W^{s,p}(B(x_0, 2^k\Lambda^2 r))}^p.
\]
Consequently we showed 
\[
 \begin{split}
 [u]^{p}_{W^{s,p}(B(x_0,r))} &\leq  2\eps\ [u]^{p}_{W^{s,p}(B(x_0, \Lambda^2 r))} + C_\eps \brac{[u]^{p}_{W^{s,p}(B(x_0, \Lambda r))} - [u]^{p}_{W^{s,p}(B( x_0,r))}}\\
  &\quad+C\ \sum_{k=1}^\infty \brac{2^{k} \Lambda}^{-\sigma} [u]_{W^{s,p}(B(x_0, 2^k\Lambda^2 r))}^p.
\end{split}
 \]
 
Absorbing $ C_\eps [u]^{p}_{W^{s,p}(B( r,x_0))}$ to the left-hand side, choosing $\Lambda$ large enough for $\tau = \frac{C_\eps + 2\eps}{C_\eps + 1} < 1$ we find
\[
 [u]^{p}_{W^{s,p}(B(x_0,r))} \leq  \tau\ [u]^{p}_{W^{s,p}(B(x_0, \Lambda^2 r))} +\sum_{k=1}^\infty \brac{2^{k} \Lambda}^{-\sigma}  [u]_{W^{s,p}(B(x_0,2^k \Lambda^2 r))}^p.
 \]
This holds for any $r \in (0,\Lambda^{-1} R)$. With an iteration argument, for details see, e.g., \cite[Lemma A.8]{Blatt-Reiter-Schikorra-2016}, one obtains $\theta > 0$ such that for any $r > 0$
\[
 [u]_{W^{s,p}(B(x_0,r))}  \aleq r^{\theta} [u]_{W^{s,p}(\R^n)}.
\]
Now since $p=\frac{n}{s}$ by Sobolev inequality on Morrey spaces, see \cite{Adams-1975}, $u$ is H\"older continuous.
\end{proof}
\section{Fractional div-curl estimates: Proof of Theorem~\ref{th:clmstypestimate} and Proposition~\ref{pr:localclms}}\label{s:hardyspaceest}
Fix a smooth nonnegative bump function $\kappa \in C_c^\infty(B(0,1))$ such that $\int_{\R^n} \kappa = 1$. Denote by $\kappa_t(x) := t^{-n} \kappa(x/t)$. The Hardy space $\mathcal{H}^1(\R^n)$ is the space of all functions $f \in L^1(\R^n)$  such that
\begin{equation}\label{eq:hardyspacedef}
 \|f\|_{\mathcal{H}^1(\R^n)} := \|\sup_{t > 0} |\kappa_t \ast f| \|_{L^1(\R^n)}<\infty.
\end{equation}
The space $BMO$ is given through the seminorm
\begin{equation}\label{eq:defbmo}
 [f]_{BMO} := \sup_{t > 0,\ x \in \R^n} t^{-n} \int_{B(x,t)} |f-(f)_{B(x,t)}|.
\end{equation}
Denote by $\mathcal{M}$ the Hardy--Littlewood maximal function 
\[
 \mathcal{M}f(x) := \sup_{t > 0} \mvint_{B(x,t)} |f|.
\]
It is well known (see, e.g., \cite{Bojarski-Hajlasz-1993, Hajlasz-1996}), that for any $p \in [1,\infty)$,
\[
 \sup_{t > 0} t^{-1} |f(x) - (f)_{B(x,t)}| \aleq \brac{\mathcal{M} |df|^p}(x)^{\frac{1}{p}}.
\]

The following is a nonlocal version of this fact.

\begin{lemma}\label{la:maximalest}
Let $s \in (0,1)$, then for any $p \in [1,\infty)$, $t > 0$,
\[
 t^{-s} |f(x) - (f)_{B(x,t)}| \aleq \brac{\int_{B(x,t)} \frac{|f(x)-f(y)|^p}{|x-y|^{n+sp}}\ dy }^{\frac{1}{p}}.
\]
\end{lemma}
\begin{proof}
This can be checked by a direct computation: For any $t > 0$,
\begin{equation}\label{eq:maximalest}
 t^{-s} |f(x) - (f)_{B(x,t)}|  \aleq\ t^{-n-s} \int_{B(x,t)} |f(x)-f(y)|\ dy.
 \end{equation}
 By Jensen's inequality for $p \geq 1$, then \eqref{eq:maximalest} can be further estimated by
 \[ 
 \begin{split}
  t^{-s} |f(x) - (f)_{B(x,t)}| &\aleq \brac{t^{-n-sp} \int_{B(x,t)} |f(x)-f(y)|^p \ dy}^{\frac{1}{p}}\\
  &= \brac{ \int_{B(x,t)} \frac{|f(x)-f(y)|^p}{|x-y|^{n+sp}} \brac{\frac{|x-y|}{t}}^{n+sp} dy}^{\frac{1}{p}}.
\end{split}
\]
Now the claim follows, since $|x-y| \leq 2t$ for any $y \in B(x,t)$.
\end{proof}

Now we give the proof of Theorem~\ref{th:clmstypestimate}, for which we adapt the argument of Coifman--Lions--Meyer--Semmes \cite[Lemma II.1]{CLMS-1993}.
\begin{proof}[Proof of Theorem~\ref{th:clmstypestimate}]
Set 
\[
 \Gamma(t,x) :=  |\kappa_t \ast (F \cdot d_s g)(x)|.
\]
We will show
\begin{equation}\label{eq:clmshardyestclaim}
\| \sup_{t > 0} \Gamma(t,\cdot)\|_{L^1(\R^n)} \aleq  \|F\|_{L^p(\Ep^1_{od} \R^n)}\ \|d_s g\|_{L^{p'}(\Ep^1_{od} \R^n)},
\end{equation}
which in view of \eqref{eq:hardyspacedef} implies the claim of Theorem~\ref{th:clmstypestimate}. 

For $t > 0$ and $x \in \R^n$, since $\div_s F = 0$ we have
\[
\begin{split}
 \Gamma(t,x)&= \kappa_t \ast (F \cdot d_s g)(x)\\
&=\int_{\R^n}\int_{\R^n} \kappa_t(x-y) F(z,y)\ (g(z)-g(y))\ \frac{dz\ dy}{|z-y|^{n+s}}\\
&=
\int_{\R^n}\int_{\R^n} F(z,y)\ \brac{\kappa_t(x-y)-\kappa_t(x-z)}\ (g(z)-(g)_{B(x,2t)})\ \frac{dz\ dy}{|z-y|^{n+s}}.
\end{split}
 \]
By the Lipschitz continuity of $\kappa$ we find,
\begin{equation}\label{eq:splitclms}
\begin{split}
 |\Gamma(t,x)|  &\aleq t^{-n-1} \int_{B(x,2t)}\int_{B(x,2t)} |g(z)-(g)_{B(x,2t)}|\,  |F(z,y)|\ \frac{dz\ dy}{|z-y|^{n+s-1}}\\
&\quad+t^{-n} \int_{B(x,t)}\int_{\R^n \backslash B(x,2t)} |g(z)-(g)_{B(x,2t)}|\,  |F(z,y)|\ \frac{dz\ dy}{|z-y|^{n+s}}\\
 &\quad+t^{-n} \int_{\R^n \backslash B(x,2t)} \int_{B(x,t)}|g(z)-(g)_{B(x,2t)}|\,  |F(z,y)|\ \frac{dz\ dy}{|z-y|^{n+s}}\\
&=: I(t,x) + II(t,x) +III(t,x).
 \end{split}
\end{equation}
As for $I(t,x)$, by H\"{o}lder inequality we obtain,
\[
\begin{split}
 I&(t,x) \\
 &\leq  t^{-n-1} \int_{B(x,2t)}\brac{\int_{B(x,2t)}\frac{|g(z)-(g)_{B(x,2t)}|^{p'}}{|z-y|^{n+p'(s-1)}}\ dy}^{\frac{1}{p'}} \brac{\int_{B(x,2t)} |F(z,y)|^{p}\frac{dy}{|z-y|^n}}^{\frac{1}{p}} \ dz.
\end{split}
 \]
Since $s < 1$, we can integrate the first term in $y$ and have
\begin{equation}\label{ineq:I}
 I(t,x)\aleq t^{-n-s} \int_{B(x,2t)}|g(z)-(g)_{B(x,2t)}| \brac{\int_{B(x,2t)} |F(z,y)|^{p}\frac{dy}{|z-y|^n} }^{\frac{1}{p}}\ dz.
\end{equation}
By H\"older inequality for some $q>1$, which will be specified below,
\begin{align}
 &\aleq t^{-n-s} \brac{\int_{B(x,2t)}|g(z)-(g)_{B(x,2t)}|^q \ dz}^{\frac{1}{q}} \brac{\int_{B(x,2t)} \brac{\int_{\R^n} |F(z,y)|^{p}\frac{dy}{|z-y|^n} }^{\frac{q'}{p}}\ dz }^{\frac{1}{q'}}\nonumber\\
 &\aleq t^{-n (\frac{n+sq}{nq})} \brac{\int_{B(x,2t)}|g(z)-(g)_{B(x,2t)}|^q\ dz}^{\frac{1}{q}} \brac{\mathcal{M} \|F\|_{p}^{q'}(x)}^{\frac{1}{q'}}\nonumber,
 \end{align}
where we recall our notation
\[
 \|F\|_{p}(y) := \brac{\int_{\R^n} |F(z,y)|^{p}\ \frac{dz}{|z-y|^n}}^{\frac{1}{p}}. 
\]
Whenever $q$ is so that
\begin{equation}\label{eq:qconds}
 q>p'\quad \mbox{and}\quad q>\frac{n}{n-s}
\end{equation}
we can apply Lemma \ref{la:locsobolev}, for $q_2 := \frac{nq}{n+sq} > 1$ and have
\[
 \brac{\int_{B(x,2t)} |g(z)-(g)_{B(x,2t)}|^{q}\,dz}^{\frac{1}{q}} \aleq \brac{\int_{B(x,\Lambda t)} \brac{\int_{\R^n} \frac{|g(z)-g(y)|^{p'}}{|z-y|^{n+p' s}} dz }^{\frac{q_2}{p'}}dy}^{\frac{1}{q_2}},
\]
for some constant $\Lambda > 1$.

Recall again, that
\[
 \|d_s g\|_{p'}(y) := \brac{\int_{\R^n} |d_s g(y,z)|^{p'}\ \frac{dz}{|z-y|^{n}} }^{\frac{1}{p'}},
\]
we obtain
\[
 I(t,x) \aleq\brac{\mathcal{M} \|d_s g\|_{p'}^{q_2}(x)}^{\frac{1}{q_2}} \brac{\mathcal{M} \|F\|_{p}^{q'}(x)}^{\frac{1}{q'}}.
\]
Thus, from H\"{o}lder inequality we get
\begin{equation}\label{eq:I}
 \begin{split}
  \int_{\R^n}\sup_{t>0} I(t,x) \,dx &\aleq \int_{\R^n} \brac{\mathcal{M} \|d_s g\|_{p'}^{q_2}(x)}^{\frac{1}{q_2}} \brac{\mathcal{M} \|F\|_{p}^{q'}(x)}^{\frac{1}{q'}}\,dx\\
  &\le \brac{\int_{\R^n} \brac{\mathcal{M} \|d_s g\|_{p'}^{q_2}(x)}^{\frac{p'}{q_2}}\ dx}^\frac{1}{p'} \brac{\int_{\R^n} \brac{\mathcal{M} \|F\|_{p}^{q'}(x)}^{\frac{p}{q'}}\ dx}^\frac{1}{p}\\
  &\le \big \|\mathcal{M}\|d_s g\|_{p'}^{q_2}\big \|_{L^{p'/q_2}(\R^n)}^{1/q_2}\ \big \|\mathcal{M}\|F\|_{p}^{q'}\big \|_{L^{p/q'}(\R^n)}^{1/q'}.
 \end{split}
\end{equation}
Now to apply the Maximal Theorem, see \cite[Theorem 1(c), p.13]{Stein-harmonic}, we choose $p/q'>1$ and $p'/q_2>1$. The latter can be always achieved for a $q$ satisfying \eqref{eq:qconds} by taking $q'$ close enough to $p'$. Therefore,
\[
\begin{split}
 \int_{\R^n}\sup_{t>0} I(t,x) \,dx &\aleq \big \|\|d_s g\|_{p'}^{q_2}\big \|_{L^{p'/q_2}(\R^n)}^{1/q_2}\ \big\|\ \|F\|_{p}^{q'} \big\|_{L^{p/q'}(\R^n)}^{1/q'}\\
 &= \|d_s g\|_{L^{p'}(\Ep^1_{od} \R^n)}\ \|F\|_{L^p(\Ep^1_{od} \R^n)}.
 \end{split}
\]
As for $II(t,x)$, applying twice H\"older inequality and using the definition of the maximal function, we obtain
\[
\begin{split}
 II&(t,x) \\&\aleq  t^{-n} \int_{B(x,t)} \brac{\int_{\R^n \backslash B(x,2t)} \frac{|g(z)-(g)_{B(x,t)}|^{p'}}{|z-y|^{n+sp'}}\, dz}^{\frac{1}{p'}}\  \brac{\int_{\R^n \backslash B(x,2t)} \frac{|F(z,y)|^p}{|z-y|^{n}}\ dz}^{\frac{1}{p}} dy\\
 &\aleq  t^{-n} \sup_{y \in B(x,t)}\brac{\int_{\R^n \backslash B(x,2t)} \frac{|g(z)-(g)_{B(x,t)}|^{p'}}{|z-y|^{n+sp'}} dz}^{\frac{1}{p'}}\  \int_{B(x,t)} \brac{\int_{\R^n \backslash B(x,2t)} \frac{|F(z,y)|^p}{|z-y|^{n}} dz}^{\frac{1}{p}} dy\\
 &\aleq  \sup_{y \in B(x,t)}\brac{\int_{\R^n \backslash B(x,2t)} \frac{|g(z)-(g)_{B(x,t)}|^{p'}}{|z-y|^{n+sp'}}\, dz}^{\frac{1}{p'}}\  \mathcal{M}\|F\|_p(x)\\
 &\aeq  \brac{\int_{\R^n \backslash B(x,2t)} \frac{|g(z)-(g)_{B(x,t)}|^{p'}}{|z-x|^{n+sp'}}\, dz}^{\frac{1}{p'}}\  \mathcal{M} \|F\|_{p}(x).\\
\end{split}
 \]
We estimate
\[
\begin{split}
 &\brac{\int_{\R^n \backslash B(x,2t)} \frac{|g(z)-(g)_{B(x,t)}|^{p'}}{|z-x|^{n+sp'}}\, dz}^{\frac{1}{p'}}\\
 &\aleq \brac{\int_{\R^n \backslash B(x,2t)} \frac{|g(z)-g(x)|^{p'}}{|z-x|^{n+sp'}}\, dz}^{\frac{1}{p'}} + \brac{\int_{\R^n \backslash B(x,2t)} \frac{|g(x)-(g)_{B(x,t)}|^{p'}}{|z-x|^{n+sp'}}\, dz}^{\frac{1}{p'}}\\
 &\aeq \brac{\int_{\R^n \backslash B(x,2t)} \frac{|g(z)-g(x)|^{p'}}{|z-x|^{n+sp'}}\, dz}^{\frac{1}{p'}} + t^{-s} |g(x)-(g)_{B(x,t)}|.
\end{split}
 \]
Thus, by Lemma~\ref{la:maximalest},
\begin{equation}\label{eq:clmsIIest}
 \brac{\int_{\R^n \backslash B(x,2t)} \frac{|g(z)-(g)_{B(x,t)}|^{p'}}{|z-x|^{n+sp'}}\, dz}^{\frac{1}{p'}} \aleq \brac{\int_{\R^n} \frac{|g(z)-g(x)|^{p'}}{|z-x|^{n+sp'}}\, dz}^{\frac{1}{p'}}.
\end{equation}
Consequently, we arrive at
\[
 \sup_{t > 0} II(t,x) \aleq \big ( \mathcal{M} \|F\|_p(x) \big )\ \brac{\int_{\R^n} \frac{|g(z)-g(x)|^{p'}}{|z-x|^{n+sp'}}\, dz}^{\frac{1}{p'}}.
\]
Integrating this, applying H\"{o}lder and then the maximal inequality, similarly as in \eqref{eq:I}, we have shown that
\[
 \|\sup_{t > 0} II(t,x)\|_{L^1(\R^n)} \aleq \|d_s g\|_{L^{p'}(\Ep^1_{od} \R^n)}\|F\|_{L^p(\Ep^1_{od} \R^n)}.
\]
Now we estimate $III(t,x)$. By H\"{o}lder inequality we find
\[
\begin{split}
 III&(t,x)\\ &= t^{-n} \int_{B(x,t)} \int_{\R^n \backslash B(x,2t)}|g(z)-(g)_{B(x,2t)}|\,  |F(z,y)|\ \frac{dz\ dy}{|z-y|^{n+s}}\\
 &\aleq  t^{-n} \int_{B(x,t)} \brac{\int_{\R^n \backslash B(x,2t)} \frac{|g(z)-(g)_{B(x,2t)}|^{p'}}{|z-y|^{n+sp'}}\, dy}^{\frac{1}{p'}}  \brac{\int_{\R^n \backslash B(x,2t)} |F(z,y)|^p\frac{dy}{|z-y|^n}}^{\frac{1}{p}} dz\\
 &\aleq  t^{-n-s} \int_{B(x,t)} |g(z)-(g)_{B(x,2t)}| \brac{\int_{\R^n} |F(z,y)|^p\frac{dy}{|z-y|^n}}^{\frac{1}{p}} dz.
\end{split}
 \]
Enlarging the set on which we integrate to $B(x,2t)$ we obtain a similar estimate as for $I$ (see \eqref{ineq:I}). Therefore, we have
\[
 \begin{split}
  \int_{\R^n}\sup_{t>0} III(t,x) \,dx \aleq \|d_s g\|_{L^{p'}(\Ep^1_{od} \R^n)}\|F\|_{L^p(\Ep^1_{od} \R^n)}.
 \end{split}
\]
Now \eqref{eq:clmshardyestclaim} is established and the proof is complete.
\end{proof}

\begin{proof}[Proof of Proposition~\ref{pr:localclms}]
By a rescaling argument, we may assume that $r = 1$ and $x_0 = 0$. We denote by $B(r)$ balls of radius $r > 0$ centered at the origin.

We need two things. Firstly, a simple H\"older inequality yields
\begin{equation}\label{eq:lclms:goal1}
  \|F \cdot d_s g\|_{L^1(B(10))} \aleq \|F\|_{L^p(\Ep^1_{od} B(10))}\ \|d_s g\|_{L^{p'}(\Ep^1_{od} B(10))}. 
\end{equation}
Secondly, let $\eta_{B(4)}$ be a nonnegative bump function constantly one on $B(4)$ and vanishing on $\R^n \backslash B(5)$. 

We claim that $\eta_{B(4)} F \cdot d_s g \in \mathcal{H}^1_{loc}(B(8))$. More precisely, set
\[
 \Gamma(t,x) :=  |\kappa_t \ast (\eta_{B(4)}\, F \cdot d_s g )(x)|,
\]
then we claim
\begin{equation}\label{eq:lclms:goal2}
 \begin{split}
   \|\sup_{t \in (0,1)} |\Gamma(t,x)|\|_{L^1(B(8))} &\aleq \|F\|_{L^p(\Ep^1_{od} B(10))}\ \|d_s g\|_{L^{p'}(\Ep^1_{od} B(10))}.
 \end{split}
\end{equation}
Assume that \eqref{eq:lclms:goal2} is proven. Let $\eta_{B(3)}$ be another nonnegative bump function constantly one on $B(3)$ and zero on $\R^n \backslash B(4)$. In particular,
\[
 \eta_{B(3)}\ F \cdot d_s g  = \eta_{B(3)}(\eta_{B(4)}\, F \cdot d_s g ).
\]
For
\[
 \lambda := \frac{\int \eta_{B(3)} F \cdot d_s g}{\int \eta_{B(3)}} 
\]
we have by  \cite[Proposition 1.92]{Semmes-1994},
\[
 \|\eta_{B(3)} (F \cdot d_s g - \lambda)\|_{\mathcal{H}^1(\R^n)} \aleq \|F \cdot d_s g\|_{L^1(B(8))} + \|\sup_{t \in (0,1)} |\Gamma(t,x)|\|_{L^1(B(8))}.
\]
In particular, for any $\varphi \in C_c^\infty(B(1))$, 
\[
 \left |\int_{\R^n} \varphi\ F \cdot d_s g \right | \aleq  [\varphi]_{BMO}\ \|\eta_{B(3)} (F \cdot d_s g - \lambda)\|_{\mathcal{H}^1(\R^n)} + |\lambda|\ \|\varphi\|_{L^1(\R^n)}.
\]
That is, once \eqref{eq:lclms:goal2} is established, we have
\[
 \left |\int_{\R^n} \varphi\ F \cdot d_s g \right | \aleq  \brac{[\varphi]_{BMO} + \|\varphi\|_{L^1(\R^n)}}\ \|F\|_{L^p(\Ep^1_{od} B(10))}\ \|d_s g\|_{L^{p'}(\Ep^1_{od} B(10))} .
\]
It remains to prove \eqref{eq:lclms:goal2}. 

We follow the strategy of the proof of Theorem~\ref{th:clmstypestimate} above. The only difference is that we have to take into account the $\eta_{B(4)}$-term. Since $F$ is divergence free,
\[
\begin{split}
 \Gamma&(t,x)\\
 &= \int_{\R^n}\int_{\R^n} \kappa_t(x-y)\, \eta_{B(4)}(y)\, F(z,y)\, \brac{g(z)-g(y)} \frac{dz\, dy}{|y-z|^{n+s}}\\
 &=\int_{\R^n}\int_{\R^n} F(z,y)\, \brac{\eta_{B(4)}(y)\, \kappa_t(x-y)-\eta_{B(4)}(z)\, \kappa_t(x-z)}\, (g(z)-(g)_{B(x,2t)})\, \frac{dz\, dy}{|y-z|^{n+s}}\\
 &=\int_{\R^n}\int_{\R^n} \eta_{B(4)}(z)\, F(z,y)\, \brac{\kappa_t(x-y)-\kappa_t(x-z)}\, (g(z)-(g)_{B(x,2t)})\, \frac{dz\, dy}{|y-z|^{n+s}}\\
 &\quad+\int_{\R^n}\int_{\R^n} F(z,y)\, \brac{\eta_{B(4)}(y)-\eta_{B(4)}(z)}\, \kappa_t(x-y)\, (g(z)-(g)_{B(x,2t)})\, \frac{dz\, dy}{|y-z|^{n+s}}\\
 &=: I(t,x) + II(t,x).
\end{split}
 \]
The first term can be treated similarly as $\Gamma(t,x)$ in the proof of Theorem~\ref{th:clmstypestimate}, and hence we obtain
\[
 \int_{B(8)}\sup_{t \in (0,1)} |I(t,x)| \aleq \|F\|_{L^p(\Ep^1_{od} B(10))}\  \|d_s g\|_{L^{p'}(\Ep^1_{od} B(10))}.
\]
For the second term, using the Lipschitz continuity of $\eta_{B(4)}$, for $t \in (0,1)$ and $x \in B(1)$, we get
\[ 
\begin{split}
II(t,x) &\aleq t^{-n}\int_{B(x,2t)} \int_{B(x,2t)} |F(z,y)|\ |g(z)-(g)_{B(x,2t)}|\ \frac{dz\ dy}{|y-z|^{n+s-1}}\\
&\quad+t^{-n} \int_{B(x,t)}\int_{\R^n \backslash B(x,2t)} |g(z)-(g)_{B(x,2t)}|\,  |F(z,y)|\ \frac{dz\ dy}{|y-z|^{n+s}}\\
&=: II_1(t,x) + II_2(t,x).
\end{split}
\]
To estimate $II_1(t,x)$ we proceed as with $I(t,x)$ in the proof of Theorem~\ref{th:clmstypestimate}, (in comparison to \eqref{eq:splitclms} we gain a $t$), and we obtain
\[
 \|\sup_{t \in (0,1)} II_1(t,\cdot) \|_{L^1(B(8))} \aleq  \|F\|_{L^p(\Ep^1_{od} B(10))}\  \|d_s g\|_{L^{p'}(\Ep^1_{od} B(10))}.
\]
For $II_2(t,x)$ we follow the estimate of $II(t,x)$ in the proof of Theorem~\ref{th:clmstypestimate},
\[
  |II_2(t,x)| \aleq  \brac{\int_{\R^n \backslash B(x,2t)} \frac{|g(z)-(g)_{B(x,t)}|^{p'}}{|z-x|^{n+sp'}}\, dz}^{\frac{1}{p'}}\  \mathcal{M}_{t < 1} \|F\|_{p}(x),\\
 \]
where 
\[
 \mathcal{M}_{t < 1} \|F\|_{p}(x) := \sup_{t \in (0,1)} \mvint_{B(x,t)} \|F\|_{p}(y)\ dy.
\]
From \eqref{eq:clmsIIest} in the proof of Theorem~\ref{th:clmstypestimate} we have 
\[
\brac{\int_{\R^n \backslash B(x,2t)} \frac{|g(z)-(g)_{B(x,t)}|^{p'}}{|z-x|^{n+sp'}}\, dz}^{\frac{1}{p'}} \aleq \brac{\int_{\R^n} \frac{|g(z)-g(x)|^{p'}}{|z-x|^{n+sp'}}\, dz}^{\frac{1}{p'}}.
 \]
Thus, by H\"older inequality
\[
 \|\sup_{t \in (0,1)} II_2(t,\cdot) \|_{L^1(B(8))} \aleq \|F\|_{L^p(\Ep^1_{od} B(10))}\ \|d_s g\|_{L^{p'}(\Ep^1_{od} B(10))}.
\]
\end{proof}

\appendix
\section{Nonlocal antisymmetric potential and the optimal gauge: Proof of Proposition~\ref{pr:genmfdEL} and Theorem~\ref{th:gauge}}
\subsection{The nonlocal antisymmetric potential: Proof of Proposition~\ref{pr:genmfdEL}}\label{s:genmfdEL}
Let $\pi: B_\delta(\mathcal{N}) \to \mathcal{N}\subset \R^N$ be the nearest point projection from a tubular neighborhood of $\mathcal{N}$ into $\mathcal{N}$. For the existence and properties of $\pi$ see, e.g., \cite{Simon-1996}. 
For $u \in \mathcal{N}$ we denote by $\Pi(u)$ the orthogonal projection onto the tangent space $T_u \mathcal{N}$. This is a symmetric matrix, and can be written as 
\[\Pi_{ij}(u) = \partial_{i} \pi^j(\pi(u)),\quad 1\leq i,j \leq N.\]
By $\Pi^\perp(u)$ we denote $I-\Pi(u)$. 

Then, for $u \in \dot{H}^{\frac{1}{2}}(\R,\mathcal{N})$ the distributional formulation of the half-harmonic map equation \eqref{eq:halfharmperp} is
\begin{equation}\label{eq:genmfdstart}
 \int_{\R} d_{\frac{1}{2}} u^i \cdot d_{\frac{1}{2}} (\Pi_{ij}(u)\varphi)(x)\ dx = 0 \quad \mbox{for any $\varphi \in C_c^\infty(\R)$}.
\end{equation}
We rewrite this equation. For $\varphi \in C_c^\infty(\R)$,
\[
\begin{split}
\int_{\R} \laph u(x)\, \varphi(x)\, dx=&\int_{\R}d_{\frac{1}{2}} u^i \cdot d_{\frac{1}{2}} \varphi (x)\ dx\\
 =&\int_{\R}d_{\frac{1}{2}} u^i \cdot d_{\frac{1}{2}} (\Pi_{ij}(u)\varphi)(x)\ dx + \int_{\R}d_{\frac{1}{2}} u^i \cdot d_{\frac{1}{2}} (\Pi_{ij}^\perp(u) \varphi)(x)\ dx.
\end{split}
 \]
The first term is zero by \eqref{eq:genmfdstart}. The second term we write in a double-integral form
\[
\begin{split}
 \int_{\R}d_{\frac{1}{2}}& u^i \cdot d_{\frac{1}{2}} (\Pi_{ij}^\perp(u) \varphi)(x)\ dx\\
 &=\int_{\R}\int_{\R} \brac{u^i(x)-u^i(y)}\, \brac{\Pi_{ij}^\perp(u(x))\, \varphi(x) - \Pi_{ij}^\perp(u(y))\, \varphi(y)}  \frac{dy\ dx}{|x-y|^{2}}\\
 &=\int_{\R}\int_{\R} \brac{u^i(x)-u^i(y)}\, \brac{\Pi_{ij}^\perp(u(x)) - \Pi_{ij}^\perp(u(y))}\, \varphi(x) \frac{dy\ dx}{|x-y|^{2}}\\
 &\quad +\int_{\R}\int_{\R} \brac{u^i(x)-u^i(y)}\, \Pi_{ij}^\perp(u(y))\ \brac{\varphi(x) -\varphi(y))}\ \frac{dy\ dx}{|x-y|^{2}}.\\
\end{split}
 \]
First we observe that the second term behaves well. For the local case, if $u \in W^{1,2}(\R,\mathcal{N})$, then $u' \in T_u \mathcal{N}$ almost everywhere. If that was also true for the fractional gradient, i.e., if we had $u(x)-u(y) \in T_{u(y)} \mathcal{N}$, then the second term above would vanish. But of course, this is in general false. However, the following simple observation provides a quantitative estimate.
\begin{lemma}\label{la:uxmuytangential}
We have \[u(x)-u(y) \in T_{u(y)} \mathcal{N} + O(|u(x)-u(y)|^2),\] more precisely,
\[
 u^i(x)-u^i(y)= \Pi_{ik}(u(y))\ (u^k(x)-u^k(y)) + O(|u(x)-u(y)|^2).
\]
\end{lemma}
In particular, we have
\[
\begin{split}
 \int_{\R}\int_{\R} \brac{u^i(x)-u^i(y)}\,& \Pi_{ij}^\perp(u(y))\  \brac{\varphi(x) -\varphi(y))}\ \frac{dy\ dx}{|x-y|^{2}}\\
 & \aleq \int_{\R}\int_{\R} |d_{\frac{1}{4}}u(x,y)|^2\ |d_{\frac{1}{2}}\varphi(x,y)| \frac{dy\ dx}{|x-y|}.
 \end{split}
\]

\begin{proof}[Proof of Lemma~\ref{la:uxmuytangential}]
Since $u(x), u(y) \in \mathcal{N}$, we have $u(x) = \pi(u(x))$, $u(y) = \pi(u(y))$. Then by the Taylor expansion of $\pi$,
\[
\begin{split}
 u^i(x) - u^i(y) &= \pi^i(u(x))-\pi^i(u(y)) \\
 &= \Pi_{ki}(u(y))\ (u^k(x)-u^k(y)) + O(|u(x)-u(y)|^2).
\end{split}
 \]
\end{proof}
So far we have shown that \eqref{eq:genmfdstart} implies
\begin{equation}\label{eq:sofarshown}
 \laph u^j = d_{\frac{1}{2}} \Pi^\perp_{ij}(u)  \cdot d_{\frac{1}{2}} u^i + R_1,
 \end{equation}
where $R_1$ satisfies the error estimate \eqref{eq:genmfd:errorestimate}. Again we use Lemma~\ref{la:uxmuytangential}
\[
 \begin{split}
 d_{\frac{1}{2}}& \Pi^\perp_{ij}(u)  \cdot d_{\frac{1}{2}} u^i(x)\\
 &=\int_{\R} \brac{\Pi^\perp_{ij}(u(x))-\Pi^\perp_{ij}(u(y))}\ \brac{u^i(x)-u^i(y)}\ \frac{dx\ dy}{|x-y|^2}\\
 &=\int_{\R} \brac{\Pi^\perp_{ij}(u(x))-\Pi^\perp_{ij}(u(y))}\ \Pi_{ik}(u(y))\, \brac{u^k(x)-u^k(y)}\ \frac{dx\ dy}{|x-y|^2}+R_2,
\end{split}
 \]
where $R_2$ again satisfies \eqref{eq:genmfd:errorestimate}, since $\Pi$ is Lipschitz and thus
\[
 |R_2(x)| \aleq \int_{\R} |u(x)-u(y)|^3\ \frac{dx\ dy}{|x-y|^2}. 
\]
Now let 
\[
 \Omega_{jk}(x,y) := \frac{\brac{\Pi^\perp_{ij}(u(x))-\Pi^\perp_{ij}(u(y))}\ \Pi_{ik} (u(y)) - \brac{\Pi^\perp_{ik}(u(x))-\Pi^\perp_{ik}(u(y))}\ \Pi_{ij} (u(y))}{|x-y|^{\frac{1}{2}}}.
\]
Clearly, $\Omega_{jk} = -\Omega_{kj} \in L^2(\Ep^1_{od} \R)$, since by assumption $d_{\frac{1}{2}} u \in L^2(\Ep^1_{od} \R)$ and $u \in L^\infty$. We conclude
\[
 \laph u^j = \Omega_{jk} \cdot d_{\frac{1}{2}} u^k + R_1 + R_2 + R_3.
 \]
Here, $R_3$ satisfies
\[
 |R_3(x)| \aleq \int_{\R} |u(x)-u(y)|^3\ \frac{dx\ dy}{|x-y|^2},
\]
by Lemma~\ref{la:uxmuytangential} and 
\[
\begin{split}
\brac{\Pi^\perp_{ik}(u(x))-\Pi^\perp_{ik}(u(y))}\ \Pi_{ij} (u(y)) & \equiv \Pi^\perp_{ik}(u(x))\ \Pi_{ij} (u(y))\\
&\equiv -(\Pi_{ij} (u(x)) - \Pi_{ij} (u(y)))\, \Pi^\perp_{ik}(u(x)). 
\end{split}
\]
This proves Proposition~\ref{pr:genmfdEL}.\qed

\subsection{Finding the optimal gauge: Proof of Theorem~\ref{th:gauge}}\label{s:gaugeconstruction}
Theorem~\ref{th:gauge} follows from the next two propositions for $s = \frac{1}{2}$ and $n=1$. The argument is an extension of the second author's \cite{Schikorra-2010}, which in turn is based on the moving frame method argument due to H\'elein \cite{Helein91}.
\begin{proposition}
Let $s \in (0,1)$, $n,N \in \N$. For any $\Omega_{ij}(x,y) \in L^2(\Ep^1_{od}\R^n)$, $i =1,\ldots,N$ there exists a minimizer $P \in \dot{H}^{s}(\R^n,SO(N))$ to
\[
\mathcal{F}(Q) := \int_{\R^n}\int_{\R^n} \left |d_sQ(x,y) - Q(x) \Omega(x,y) \right |^2\ \frac{dx\ dy}{|x-y|^n} , \quad Q \in \dot{H}^{s}(\R^n,SO(N)).
\]
This minimizer satisfies
\[
 \|P\|_{\dot{H}^{s}(\R^n)} \aleq \|\Omega\|_{L^2(\Ep^1_{od}\R^n)}.
\]
\end{proposition}
\begin{proof}
We have
\[
 [Q]_{\dot{H}^s(\R^n,\R^{N \times N})} \leq \mathcal{F}(Q) + \|\Omega\|_{L^2(\Ep^1_{od}\R^n)}.
\]

Let $P_k$ be a minimizing sequence in $\dot{H}^{s}(\R^n,SO(N))$. Since $P_k$ maps into $SO(N)$ pointwise a.e., it is, in particular, bounded. Thus, up to a subsequence, we can assume that $P_k$ converges to some $P$ weakly in $\dot{H}^{s}(\R^n)$, strongly in $L^2_{loc}(\R^n)$ and pointwise almost everywhere. Thus $P \in \dot{H}^s(\R^n,SO(N))$. 

The Lemma of Fatou, simply by pointwise a.e. convergence of $P_k$ to $P$, implies
\[
 \mathcal{F}(P) \leq \liminf_{k \to \infty} \mathcal{F}(P_k).
\]
Thus, $P$ is indeed a minimizer. The norm estimate follows since $Q \equiv I_{N \times N}$ is admissible.
\end{proof}

\begin{proposition}
Let $s \in (0,1)$, $n,N \in \N$. Let $P \in \dot{H}^{s}(\R^n,SO(N))$ be a critical point of $\mathcal{F}$ in the class of maps in $\dot{H}^{s}(\R^n,SO(N))$. Then, for $\Omega^P \in L^2(\Ep^{1}_{od} \R^n)$ given by
\[
 \Omega^P(x,y) =\frac{1}{2} \brac{d_sP(x,y) \brac{P^T(y) + P^T(x)} - P(x) \Omega(x,y) P^T(y) + P(y) \Omega^T(x,y) P^T(x)} 
\]
it holds 
\[
\div_s \Omega^P = 0. 
\]
\end{proposition}
Observe that if $\Omega_{ij}(x,y) = -\Omega_{ji}(x,y)$ almost everywhere, then 
\[
 \Omega^P(x,y) = \frac{1}{2} \brac{d_sP(x,y) \brac{P^T(y) + P^T(x)} - P(x) \Omega(x,y) P^T(y) - P(y) \Omega(x,y) P^T(x)}.
\]
\begin{proof}
Let $P \in  L^\infty \cap \dot{H}^{s}(\R^n,SO(N))$ be a critical point of $\mathcal{F}$.

To compute the Euler--Lagrange equation, define for some $\varphi \in C_c^\infty(\R^n)$ and a constant $\alpha \in so(N)$ the variation 
\[
 P_t(x) := e^{t \alpha \varphi(x)} P(x),\quad t \in \R.
\]
Clearly, $P_t \in \dot{H}^s(\R^n,SO(N))$ and $P_0 = P$, so $P_t$ is an admissible variation of $P$. Moreover,
\[
 \frac{d}{dt} P_t(x) = \varphi(x)\ \alpha\, P(x)
\]
and, since $P$ is critical,
\begin{equation}\label{eq:PEL}
\begin{split}
 0 &= \frac{d}{dt} \Big |_{t = 0}  \mathcal{F}(P_t)\\
 &= 2\int_{\R^n}\int_{\R^n} \brac{d_sP(x,y) - P(x)\, \Omega(x,y)}: \frac{d}{dt} \Big|_{t=0}\brac{d_sP_t(x,y) - P_t(x) \Omega(x,y)} \ \frac{dx\ dy}{|x-y|^n}.
\end{split}
 \end{equation}
Here $A:B := \sum_{i,j=1}^N A_{ij}B_{ij}$ is the Hilbert--Schmidt scalar product for matrices. We compute
\[
 \begin{split}
 \frac{d}{dt} \Big|_{t=0} d_sP_t(x,y) - P_t(x) \Omega(x,y)
 =\alpha\, d_s (\varphi\, P)(x,y)-  \alpha\, P(x)\, \Omega(x,y)\, \varphi(x).
 \end{split}
\]
Now
\[
 d_s (\varphi\, P)(x,y)  =  P(y)\, d_s \varphi(x,y)  + d_s P(x,y)\, \varphi(x),
\]
and we arrive at
\[
 \begin{split}
 \frac{d}{dt} \Big|_{t=0} d_s P_t(x,y) &- P_t(x) \Omega(x,y)\\
 &=\alpha \brac{d_s P(x,y) - P(x) \Omega(x,y)}\ \varphi(x) +  \alpha P(y)\, d_s \varphi(x,y).
 \end{split}
\]
Since $\alpha \in so(N)$, we have
\[
 \brac{d_s P(x,y) - P(x) \Omega(x,y)}: \alpha \brac{d_s P(x,y) - P(x) \Omega(x,y)} \equiv 0.
\]
Thus \eqref{eq:PEL} can be rewritten as
\[
\begin{split}
 0 &= \int_{\R^n}\int_{\R^n} \brac{d_sP(x,y) - P(x) \Omega(x,y)}: \alpha P(y)\, d_s \varphi(x,y) \ \frac{dx\ dy}{|x-y|^n}\\
 &= \int_{\R^n}\int_{\R^n} \brac{d_sP(x,y)P^T(y)  - P(x) \Omega(x,y) P^T(y)}: \alpha \, d_s \varphi(x,y) \ \frac{dx\ dy}{|x-y|^n}.
\end{split}
 \]
This holds for any antisymmetric $\alpha \in so(N)$, so we have componentwise
\[
 \div_s \Omega^P = 0,
\]
for 
\[
\begin{split}
 \Omega^P(x,y) &= \operatorname{so}\brac{d_sP(x,y) P^T(y) - P(x) \Omega(x,y) P^T(y)}\\
 &=\frac{1}{2} \brac{d_sP(x,y) \brac{P^T(y) + P^T(x)} - P(x) \Omega(x,y) P^T(y) + P(y) \Omega^T(x,y) P^T(x)}.\\ 
\end{split}
 \]
\end{proof}

\section{Euler--Lagrange equations for \texorpdfstring{$W^{s,p}$}{W(s,p)}-harmonic maps into homogeneous Riemannian manifolds: Proof of Lemma~\ref{la:homoEL}}\label{s:homoEL}
Let the $W^{s,p}$-energy $\mathcal{E}_{s,p}$ be given by
\[
 \mathcal{E}_{s,p}(v) := [v]_{W^{s,p}(\R^n)}^p.
\]
Then $W^{s,p}$-harmonic maps into a smooth, compact manifold $\mathcal{N}$ without boundary are maps that satisfy the Euler--Lagrange equation of $\mathcal{E}_{s,p}$ with the side condition $v(x) \in \mathcal{N}$ for almost every $x \in \R^n$. Namely, a $W^{s,p}$-harmonic map into $\mathcal{N}$ is a distributional solution to
\[
\div_s (|d_s u|^{p-2} d_s u) \perp T_u\mathcal{N}.
\]
Recall from section \ref{s:genmfdEL} that we denote by $\Pi: \mathcal{N} \to \R^{N \times N}$ the projection onto the tangent plane $T\mathcal{N}$. Then $u$ is $W^{s,p}$-harmonic if and only if
\begin{equation}\label{eq:EL}
 \int_{\R^n} \brac{|d_s u|^{p-2} d_su^i} \cdot d_s (\Pi_{ij}(u) \varphi)(x)\ dx = 0 \quad \mbox{for all }\varphi \in C_c^\infty(\R^n) \quad \mbox{for $i=1,\ldots,N$}.
\end{equation}
Similarly to \eqref{eq:sofarshown} we find
\begin{lemma}\label{la:ELeq}
Let $u$ be a $W^{s,p}$-harmonic map into $\n$. Then,
\begin{align*}
 &\int_{\R^n} \brac{|d_s u|^{p-2} d_su^j} \cdot d_s \varphi(x)\ dx\\
 &=\int_{\R^n} \int_{\R^n} |d_s u(x,y)|^{p-2} d_s u^i(x,y)\ d_s\Pi_{ij}(u)(x,y)\ \varphi(x)\ \frac{dx\ dy}{|x-y|^{n}}+ \int_{\R^n} R\ \varphi.
\end{align*}
Here, the error term $R$ satisfies
\[
\left |\int_{\R^n} R\ \varphi \right | \aleq  \int_{\R^n}\int_{\R^n} |d_{\frac{s}{p'}}u(x,y)|^p\ |d_{s}\varphi(x,y)|\ \frac{dy\, dx}{|x-y|^n}.
\]
\end{lemma}
From now on throughout the section we assume that $\n$ is a homogeneous Riemannian manifold which is equivariantly embedded into $\R^N$.
We recall that a homogeneous space $\n$ is a quotient space $G/H$, where $G$ is a connected Lie group and $H$ is a closed subgroup. Let $m=\dim(G)$, with the help of \cite{Freire} or \cite[Lemma 2]{Helein91-sym}, we find a family of $G$-Killing vector fields $\{X_\alpha\}_{\alpha=1}^m$ on $\n$ and a family $\{Y_\alpha\}_{\alpha=1}^m$ of smooth tangent vector fields, such that for any $y\in\n$
\[
 v = \sum_{\alpha=1}^m \left\langle X_\alpha, v\right\rangle Y_\alpha \quad  \text{ for any } v\in T_y\n.
\]
We will use the following two properties of Killing fields: Firstly, from equation (19) in \cite{Schikorra-Sire-Wang-2015} we have the following nonlocal Killing field property
\begin{equation}\label{eq:killingprop}
\left\langle X_\alpha(p) - X_\alpha(q),p-q\right\rangle =0, \quad \text{ for all } p,\ q\in\n,\ 1\le\alpha\le m.
\end{equation}
Secondly, we have that the projection $\Pi_{ij}(p)$ into the tangent plane $T_p \mathcal{N}$ can be written as
\begin{equation}\label{eq:killingprojection}
 \Pi_{ij}(p) = \sum_{\alpha=1}^m X^i_\alpha(p)Y^j_\alpha(p) = \sum_{\alpha=1}^m Y_\alpha^i(p)X^j_\alpha(p),\quad \text{ for all } p\in\n.
\end{equation}
New we follow the arguments of the local case, see \cite{Helein91-sym} and \cite[pp.90--91]{Toro-Wang}, to rewrite the Euler--Lagrange equations in such a way that a fractional div-curl quantity appears.  
\begin{proof}[Proof of Lemma \ref{la:homoEL}]
By Lemma \ref{la:ELeq} and \eqref{eq:killingprojection}, any $W^{s,p}$-harmonic map $u$ satisfies
\[
\int_{\R^n} |d_su|^{p-2}d_su^j\cdot d_s \varphi = \int_{\R^n} \int_{\R^n} |d_su|^{p-2}d_su^i \ d_s \brac{X_\alpha^i(u)Y_\alpha^j(u)}(x,y)\ \varphi(x) \frac{dx\ dy}{|x-y|^n}
+ \int_{\R^n} R\ \varphi.
\]
From property \eqref{eq:killingprop} for each $\alpha$,
\begin{equation}\label{eq:kill1}
 d_s u^i(x,y)\ d_s X_\alpha^i(u)(x,y) = 0.
\end{equation}
Thus,
\begin{equation}\label{eq:kill2}
  d_s u^i(x,y) \ d_s \brac{X_\alpha^i(u)Y_\alpha^j(u)}(x,y) = X_\alpha^i(u(x)) \ d_s u^i(x,y) \ d_sY_\alpha^j(u)(x,y).
\end{equation}
Plugging this in
\[
 \begin{split}
  &\int_{\R^n} \int_{\R^n} |d_su|^{p-2}d_su^i \, d_s \brac{X_\alpha^i(u)Y_\alpha^j(u)}(x,y)\, \varphi(x) \frac{dx\, dy}{|x-y|^n}\\
  &=  \int_{\R^n} \int_{\R^n}X_\alpha^i(u(x))\varphi(x) |d_s u(x,y)|^{p-2} \, d_s u^i(x,y)\, d_sY_\alpha^j(u)(x,y)\, \frac{dx\, dy}{|x-y|^n}\\
  &= \frac{1}{2}\int_{\R^n} \int_{\R^n}\varphi(x)\, |d_s u(x,y)|^{p-2} \, d_s u^i(x,y)\brac{X_\alpha^i(u(x)) + X_\alpha^i(u(y))}\, d_sY_\alpha^j(u)(x,y)\, \frac{dx\, dy}{|x-y|^n}\\
  &\quad + \frac{1}{2}\int_{\R^n} \int_{\R^n} \varphi(x)\, |d_s u(x,y)|^{p-2} \, d_s u^i(x,y)\brac{X_\alpha^i(u(x)) - X_\alpha^i(u(y))}\, d_sY_\alpha^j(u)(x,y)\, \frac{dx\, dy}{|x-y|^n}.
\end{split}
\]
The last term is zero, again by \eqref{eq:kill2}. We set
\begin{equation}\label{eq:defomega}
 \Omega_\alpha (x,y) := \frac{1}{2} \brac{X_\alpha^i(u(x)) + X_\alpha^i(u(y))} d_s u^i (x,y).
\end{equation}
Consequently,
\[
\int_{\R^n} |d_su|^{p-2}d_su^j\cdot d_s \varphi = \int_{\R^n} \int_{\R^n}\varphi(x)\, |d_s u(x,y)|^{p-2} \ \Omega_\alpha (x,y)\ d_sY_\alpha^j(u)(x,y)\ \frac{dx\ dy}{|x-y|^n}
+ \int_{\R^n} R\ \varphi.
\]
This proves Lemma~\ref{la:homoEL} up to showing that $|d_s u|^{p-2} \ \Omega_\alpha$ is divergence free. The latter is contained in the following lemma.
\end{proof}

\begin{lemma}
Let $u$, $\mathcal{N}$, $\Omega_\alpha$ be as above. Then
\[
 \div_s \brac{|d_s u|^{p-2} \ \Omega_\alpha} = 0.
\]
\end{lemma}
\begin{proof}
For any test-function $\psi \in C_c^\infty(\R^n)$ setting $\varphi^j := X^j_\alpha(u)\psi$ we have $\varphi^i \equiv \Pi_{ji}(u) \varphi^j$, because $X_\alpha(u)\psi$ is a tangent field. 
From equation \eqref{eq:EL} we then have
\begin{equation}\label{eq:EL22}
 \int_{\R^n} \brac{|d_s u|^{p-2} d_su^i} \cdot d_s \varphi^i(x)\ dx = 0.
\end{equation}
Now, with \eqref{eq:kill1} pointwise almost everywhere,
\[
 \Omega_\alpha(x,y)\, d_s \psi(x,y) = d_s u^i(x,y)\ d_s \brac{X_\alpha^i(u)\ \psi}(x,y) = d_s u^i(x,y)\ d_s \varphi^i(x,y).
\]
Thus, with \eqref{eq:EL22},
\[
 \int_{\R^n} |d_s u|^{p-2} \Omega_\alpha\cdot d_s \psi(x)\ dx = \int_{\R^n} \brac{|d_s u|^{p-2} d_su^i} \cdot d_s \varphi^i(x)\ dx = 0.
\]
This completes the proof. 
\end{proof}

\section{An integro-differential Triebel--Lizorkin type space}\label{s:triebel}
The Sobolev space $W^{s,2}(\R^n)$, $s \in (0,1)$, $p \in (1,\infty)$ is equivalent to the Triebel--Lizorkin spaces $F^s_{p,2}(\R^n)$, see \cite{GrafakosCF,GrafakosMF,Runst-Sickel-1996}. More precisely,
\begin{equation}\label{eq:triebelFsp2}
 \|f\|_{\dot{F}^s_{p,2}} \aeq \|\laps{s} f\|_{L^p(\R^n)}.
\end{equation}
More generally, the Sobolev space $W^{s,p}(\R^n)$, $s \in (0,1)$, $p \in (1,\infty)$ is equivalent to the Triebel--Lizorkin space $F^s_{p,p}(\R^n)$, and we have
\[
 \|f\|_{\dot{F}^s_{p,p}} \aeq \brac{\int_{\R^n} \int_{\R^n} \frac{|f(x)-f(y)|^p}{|x-y|^{n+sp}}\ dy\ dx}^{\frac{1}{p}}.
\]
For $p,q \in (1,\infty)$, $s \in (0,1)$ we introduce the space $\dot{X}^{s}_{p,q}$ induced by the seminorm
\[
\|f\|_{\dot{X}^s_{p,q}} := 
\brac{\int_{\R^n} \brac{\int_{\R^n} \frac{|f(x)-f(y)|^q}{|x-y|^{n+sq}}\ dy}^{\frac{p}{q}}\ dx}^{\frac{1}{p}}.
\]
We have the following embedding
\begin{proposition}\label{pr:triebelembedding}
For any $s \in (0,1)$, $p,q \in (1,\infty)$, we have for the homogeneous Triebel--Lizorkin space $\dot{F}^s_{p,\infty}$
\[
 \|f\|_{\dot{F}^s_{p,\infty}} \aleq \brac{\int_{\R^n} \brac{\int_{\R^n} \frac{|f(x)-f(y)|^q}{|x-y|^{n+sq}}\ dy}^{\frac{p}{q}}\ dx}^{\frac{1}{p}}.
\]
\end{proposition}
\begin{remark}
In particular, this proposition shows that $\dot{X}^{s}_{p,q}$ is not the Besov space $\dot{B}^s_{p,q}$, at least for $q>p$. Indeed, otherwise the embedding above would imply $\dot{B}^s_{p,q} \subset \dot{F}^{s}_{p,\infty}$, which is false for $q > p$. 
One could think that $\dot{X}^s_{p,q} = \dot{F}^{s}_{p,q}$, however we were not able to immediately prove (or disprove) this. \end{remark}

\begin{proof}
In the arxiv-version of \cite{Schikorra-CPDE}, in Proposition~8.6, one can find the (easy) proof of
\[
 2^{js} |\lap_j f(x)| \aleq \brac{\int_{\R^n} \frac{|f(x)-f(y)|^q}{|x-y|^{n+sq}}\ dy}^{\frac{1}{q}}.
\]
Here, $\lap_j$ is the $j$-th Littlewood--Paley projection operator, see \cite{GrafakosCF,GrafakosMF}. Thus,
\[
 \|f\|_{\dot{F}^s_{p,\infty}} \equiv \brac{\int_{\R^n} \brac{\sup_{j \in \Z} 2^{js} |\lap_jf(x)|}^p\ dx}^{\frac{1}{p}} \aleq 
 \brac{\int_{\R^n}\brac{\int_{\R^n} \frac{|f(x)-f(y)|^q}{|x-y|^{n+sq}}\ dy}^{\frac{p}{q}}\ dx}^{\frac{1}{p}}.
\]
\end{proof}
For the Triebel--Lizorkin spaces $\dot{F}^s_{p,\infty}$ we have a Sobolev embedding (see, e.g., \cite[Theorem~2.7.1 (ii)]{Triebel1983}) \[\dot{F}^{s_1}_{p_1,q_1}(\R^n) \hookrightarrow  \dot{F}^{s_2}_{p_2,q_2}(\R^n)\] holds whenever $s_1 > s_2$ and for any $q_1,q_2 \in [1,\infty]$ if $p_1, p_2 \in (1,\infty)$ are such that \[s_1 - \frac{n}{p_1} = s_2 - \frac{n}{p_2}.\] From this, \eqref{eq:triebelFsp2}, and Proposition~\ref{pr:triebelembedding} we have in particular the following Sobolev-type inequality, cf. \cite[Theorem 1.6]{Schikorra-CPDE}.
\begin{proposition}\label{pr:sobolevin}
Let $0 \leq s_2 < s_1 < 1$ and $p_1,p_2,q_1\in (1,\infty)$ so that
\[
 s_1 - \frac{n}{p_1} = s_2 - \frac{n}{p_2}.
\]
Then
 \begin{equation}\label{eq:sobolev}
  \|\laps{s_2} f \|_{L^{p_2}(\R^n)} \aleq \brac{\int_{\R^n}\brac{\int_{\R^n} \frac{|f(x)-f(y)|^{q_1}}{|x-y|^{n+s_1 q_1}}\, dx}^{\frac{p_1}{q_1}}\, dy}^{\frac{1}{p_1}}.
  \end{equation}
\end{proposition}
From Proposition~\ref{pr:sobolevin} we also obtain the following localized version.
\begin{lemma}\label{la:locsobolev}
Let $s \in (0,1)$ and $p,p_2 \in (1,\infty)$. For any $q \in (1,p_2)$, if
\begin{equation}\label{eq:lsp2def}
 s - \frac{n}{p} = -\frac{n}{p_2}
\end{equation}
then there is some $\Lambda > 0$ so that for any ball $B(x_0,\rho)$ we have
\[
  \|f - (f)_{B(x_0,\rho)}\|_{L^{p_2}(B(x_0,\rho))} \leq C \brac{\int_{B(x_0,\Lambda \rho)}\brac{\int_{B(x_0,\Lambda \rho)} \frac{|f(x)-f(y)|^{q}}{|x-y|^{n+s q}}\, dx}^{\frac{p}{q}}\, dy}^{\frac{1}{p}}.
\]
Here $(f)_{B(x_0,\rho)}$ denotes the mean value. The constant $C$ depends only on $s,p,p_2,q$ and the dimension.
\end{lemma}
\begin{proof}
By translation invariance and scaling we may assume that $\rho = 1$ and $x_0 = 0$. We denote balls centered at the origin by $B(r) := B(0,r)$, for $r > 0$.

Let $\eta \in C_c^\infty(B(2))$ be a typical cutoff function, $\eta\equiv1$ on $B(1)$, $\|\nabla\eta \|_{L^\infty} \aleq 1$. We apply Proposition~\ref{pr:sobolevin} to \[u(x) := \eta(x)(f(x)-(f)_{B(1)}).\] Then,
 \[
  \begin{split}
   \|u\|_{L^{p_2}({\R^n})} &\aleq \brac{\int_{\R^n}\brac{\int_{\R^n} \frac{|\eta(x)(f(x)-(f)_{B(1)})-\eta(y)(f(y)-(f)_{B(1)})|^q}{|x-y|^{n+sq}}\, dx}^{\frac{p}{q}}\, dy}^{\frac{1}{p}}\\
  &\aleq I + II + III,
  \end{split}
 \]
where
\[
    I := \brac{\int_{B(\Lambda )}\brac{\int_{B(\Lambda )} \frac{|\eta(x)(f(x)-(f)_{B(1)})-\eta(y)(f(y)-(f)_{B(1)})|^q}{|x-y|^{n+sq}}\, dx}^{\frac{p}{q}}\, dy}^{\frac{1}{p}},
\]
\[
    II := \brac{\int_{B(\Lambda )}\brac{\int_{\R^n\setminus B(\Lambda )} \frac{|\eta(y)(f(y)-(f)_{B(1)})|^q}{|x-y|^{n+sq}}\, dx}^{\frac{p}{q}}\, dy}^{\frac{1}{p}},
\]
and
\[   
   III:= \brac{\int_{\R^n\setminus B(\Lambda )}\brac{\int_{B(2)} \frac{|\eta(x)(f(x)-(f)_{B(1)})|^q}{|x-y|^{n+sq}}\, dx}^{\frac{p}{q}}\, dy}^{\frac{1}{p}}.
\]
As for $I$ we estimate
\[
 \begin{split}
  I&\aleq 
   \brac{\int_{B(\Lambda )}\brac{\int_{B(\Lambda )} \frac{|(\eta(x)-\eta(y))(f(x)-f(y))|^q}{|x-y|^{n+sq}}\, dx}^{\frac{p}{q}}\, dy}^{\frac{1}{p}}\\
&\quad+  \brac{\int_{B(\Lambda )}\brac{\int_{B(\Lambda )} \frac{|\eta(y)(f(x)-f(y))|^q}{|x-y|^{n+sq}}\, dx}^{\frac{p}{q}}\, dy}^{\frac{1}{p}}\\
 &\quad + \brac{\int_{B(\Lambda )}\brac{\int_{B(\Lambda )} \frac{|(\eta(x)-\eta(y))(f(y)-(f)_{B(1)})|^q}{|x-y|^{n+sq}}\, dx}^{\frac{p}{q}}\, dy}^{\frac{1}{p}}.
\end{split}
\]
Regrouping and using the Lipschitz continuity of $\eta$ we find
\[
\begin{split}
 I &\aleq \brac{\int_{B(\Lambda )}\brac{\int_{B(\Lambda )} \frac{|f(x)-f(y)|^q}{|x-y|^{n+sq}}\, dx}^{\frac{p}{q}}\, dy}^{\frac{1}{p}}\\
  &\quad + \brac{\int_{B(\Lambda )}\brac{\int_{B(\Lambda )} \frac{|f(y)-(f)_{B(1)}|^q}{|x-y|^{n+(s-1)q}}\, dx}^{\frac{p}{q}}\, dy}^{\frac{1}{p}}.
 \end{split}
\]
Since $s<1$ we can integrate in $x$ and have
\[
 \begin{split}
 &\brac{\int_{B(\Lambda )}\brac{\int_{B(\Lambda )} \frac{|f(y)-(f)_{B(1)}|^q}{|x-y|^{n+(s-1)q}}\, dx}^{\frac{p}{q}}\, dy}^{\frac{1}{p}}\\
 &\aleq \brac{\int_{B(\Lambda )}|f(y)-(f)_{B(1)}|^{p}\, dy}^{\frac{1}{p}}\\ 
 &\aleq \brac{\int_{B(\Lambda )}\brac{\int_{B(1)} |f(y)-f(x)|\,dx }^{p}\,dy}^{\frac{1}{p}}\\
 &\aleq \brac{\int_{B(\Lambda )}\brac{\int_{B(1)} \frac{|f(y)-f(x)|^q}{|x-y|^{n+sq}}\,dx }^{\frac{p}{q}}\,dy}^{\frac{1}{p}}.
 \end{split}
\]
That is, 
\[
 I \aleq  \brac{\int_{B(\Lambda )}\brac{\int_{B(\Lambda )} \frac{|f(x)-f(y)|^q}{|x-y|^{n+sq}}\, dx}^{\frac{p}{q}}\, dy}^{\frac{1}{p}}.
\]
As for $II$, we integrate in $x$ and have
 \[
   II \aleq \brac{\int_{B(2)}|f(y)-(f)_{B(1)}|^{p}\, dy}^{\frac{1}{p}}
  \aleq \brac{\int_{B(2)}\brac{\int_{B(1)} \frac{|f(x)-f(y)|^q}{|x-y|^{n+sq}}\ dx}^\frac{p}{q}\ dy}^\frac{1}{p}.
\]
Finally we estimate $III$. For $\Lambda > 4$, for any $x\in {B(2)}$ and any $y\in \R^n\setminus B(\Lambda)$ we have  $|x-y|\ge \frac{1}{2}|y|$ and we get
\[
III \aleq \brac{\int_{\R^n\setminus B(\Lambda )}|y|^{-(n+sq)\frac{p}{q}}\brac{\int_{{B(2)}} |{\eta(x)}(f(x)-(f)_{{B(1)}})|^q\, dx}^{\frac{p}{q}}\, dy}^{\frac{1}{p}}.
\]
Now, since $q \in (1,p_2)$ we have $-(n+sq)\frac{p}{q} < -n$. Hence,
\[
 \brac{\int_{\R^n\setminus B(\Lambda )}|y|^{-(n+sq)\frac{p}{q}}\ dy}^{\frac{1}{p}} \aleq \Lambda^{-\sigma},
\]
for $\sigma := \frac{n}{q}+-\frac{n}{p}+s >0$. Thus, by the definition of $u$,
\[
\begin{split}
III &\aleq \Lambda^{-\sigma}\brac{\int_{{B(2)}}|{\eta(x)}(f(x)-(f)_{{B(1)}})|^q\, dx}^{\frac1q}\\
&\aleq \Lambda^{-\sigma}\, \norm{u}_{L^q(B(2))}\aleq \Lambda^{-\sigma}\, \norm{u}_{L^{p_2}(\R^n)}.
\end{split}
\]
In the last inequality we used the fact that $q < p_2$ and $u$ is supported in $B(2)$.
We have thus shown, for any $\Lambda > 4$,
\[
\|u\|_{L^{p_2}({\R^n})} \aleq   \Lambda^{-\sigma} \norm{u}_{L^{p_2}(\R^n)} + \brac{\int_{B(\Lambda)}\brac{\int_{B(\Lambda)} \frac{|f(x)-f(y)|^q}{|x-y|^{n+sq}}\ dx}^\frac{p}{q}\ dy}^\frac{1}{p}.
\]
For $\Lambda$ sufficiently large we can absorb $\Lambda^{-\sigma} \norm{u}_{L^{p_2}(\R^n)}$ into the left-hand side. This finishes the proof of Lemma~\ref{la:locsobolev}.
\end{proof}
\subsection*{Acknowledgment}
We would like to thank M. Hinz for pointing out literature regarding Dirichlet forms and the subsequent gradients and divergence.

Both authors are supported by the German Research Foundation (DFG) through grant no.~SCHI-1257-3-1. A.S. receives funding from the Daimler and Benz foundation. A.S. is Heisenberg fellow.  
 \bibliographystyle{plain}%
\bibliography{bib}%
\end{document}